\documentclass[11pt]{article}
\usepackage{amssymb}
\usepackage{amsmath}
\usepackage{amsfonts}

\newtheorem{theorem}{Theorem}

\newtheorem{claim}[theorem]{Claim}

\newtheorem{definition}[theorem]{Definition}

\newtheorem{lemma}[theorem]{Lemma}

\newtheorem{proposition}[theorem]{Proposition}
\newtheorem{remark}[theorem]{Remark}

\newenvironment{proof}[1][Proof]{\noindent\textbf{#1.} }{\ \rule{0.5em}{0.5em}}
\input{tcilatex}
\begin{document}

\begin{center}
\renewcommand{\thefootnote}{} \footnotetext{%
The project was supported by the Research Center, College of Science, King
Saud University.
}{\large CLASSIFICATION OF COMPLETE LEFT-INVARIANT
AFFINE STRUCTURES ON THE OSCILLATOR GROUP}

\bigskip\ 

\bigskip MOHAMMED GUEDIRI

\vspace{0.5in}
\end{center}

$\mathtt{Abstract.}$\emph{\ The goal of this paper is to provide a method,
based on the theory of extensions of left-symmetric algebras, for
classifying left-invariant affine structures on a given solvable Lie group
of low dimension. To better illustrate our method, we shall apply it to
classify all complete left-invariant affine structures on the oscillator
group. }

\section{Introduction}

It is a well known result (see \cite{auslander}, \cite{milnor}) that a
simply connected Lie group $G$ which admits a complete left-invariant affine
structure, or equivalently $G$ acts simply transitively by affine
transformations on $\mathbb{R}^{n},$ must be solvable. It is also well known
that not every solvable (even nilpotent) Lie group can admit an affine
structure \cite{benoist}. On the other hand, given a simply connected
solvable Lie group $G$ which can admit a complete left-invariant structure,
it is important to classify all such possible structures on $G$.

\smallskip

Our goal in the present paper is to provide a method for classifying
left-invariant affine structures on a given solvable Lie group of low
dimension. Since the classification has been completely achieved up to
dimension four in the nilpotent case (see \cite{friedgold}, \cite{kim}, \cite%
{kuiper}), we shall illustrate our method by applying it to the
classification of all complete left-invariant affine structures on the
remarkable solvable non-nilpotent $4$-dimensional Lie group $O_{4}$, known
as the \emph{oscillator group}. Recall that $O_{4}$ can be viewed as a
semidirect product of the real line with the Heisenberg group. Recall also
that the Lie algebra $\mathcal{O}_{4}$ of $O_{4}$ (that we shall call \emph{%
oscillator algebra}) is the Lie algebra with generators $%
e_{1},e_{2},e_{3},e_{4},$ and with nonzero brackets 
\begin{equation*}
\lbrack e_{1},e_{2}]=e_{3},\ [e_{4},e_{1}]=e_{2},\ \left[ e_{4},e_{2}\right]
=-e_{1}.
\end{equation*}

Since left-invariant affine structures on a Lie group $G$ are in one-to-one
correspondence with left-symmetric structures on its Lie algebra $\mathcal{G}
$ \cite{kim}, we shall carry out the classification of complete
left-invariant affine structures on $O_{4}$ in terms of complete (in the
sense of \cite{segal}) left-symmetric structures on $\mathcal{O}_{4}$.

\bigskip

The paper is organized as follows. In Section 2, we shall recall the notion
of extensions of Lie algebras and its relationship to the notion of $%
\mathcal{G}$-kernels. In Section 3, we shall give some necessary
definitions, notations, and basic results on left-symmetric algebras and
their extensions. In Section 4, we shall consider a complete real
left-symmetric algebra $A_{4}$ whose associated Lie algebra is the
oscillator algebra $\mathcal{O}_{4}$. Using the complexification of $A_{4}$
and some results in \cite{burde} and \cite{kbm} we shall first show that $%
A_{4}$ is not simple. Precisely, we shall show that $A_{4}$ has a proper
two-sided ideal whose associated Lie algebra is isomorphic to the center $%
Z\left( \mathcal{O}_{4}\right) \cong \mathbb{R}$ or the commutator ideal $%
\left[ \mathcal{O}_{4},\mathcal{O}_{4}\right] \cong \mathcal{H}_{3}$ of $%
\mathcal{O}_{4}.$ In the latter case, we shall show that the so-called
center of $A_{4}$ is nontrivial, and therefore we can get $A_{4}$ as a
central (in some sense that will be defined later) extension of a complete $%
3 $-dimensional left-symmetric algebra $A_{3}$ by the trivial left-symmetric
algebra $\mathbb{R}$ (i.e., the vector space $\mathbb{R}$ with the trivial
left-symmetric product).

In Section 5, we shall show that in both cases we have a short exact
sequence (which turns out to be central) of left-symmetric algebras of the
form 
\begin{equation*}
0\rightarrow \mathbb{R}\overset{i}{\rightarrow }A_{4}\overset{\pi }{%
\rightarrow }A_{3}\rightarrow 0,
\end{equation*}%
where $A_{3}$\ is a complete left-symmetric algebra whose Lie algebra is
isomorphic to the Lie algebra$\mathcal{\ E}\left( 2\right) \ $of the group
of Euclidean motions of the plane. We shall then show that, up to
left-symmetric isomorphism, there are only two non-isomorphic complete
left-symmetric structures on $\mathcal{E}\left( 2\right) $, and we shall use
these to carry out all complete left-symmetric structures on $\mathcal{O}%
_{4}.$ We shall see that one of these two left-symmetric structures on $%
\mathcal{E}\left( 2\right) $ yields exactly one complete left-symmetric
structure on $\mathcal{O}_{4}.$ However, the second one yields a
two-parameter family of complete left-symmetric algebras $A_{4}\left(
s,t\right) $ whose associated Lie algebra is $\mathcal{O}_{4},$ and the
conjugacy class of $A_{4}\left( s,t\right) $ is given as follows: $%
A_{4}\left( s^{\prime },t^{\prime }\right) $ is isomorphic to $A_{4}\left(
s,t\right) $ if and only if $\left( s^{\prime },t^{\prime }\right) =\left(
\alpha s,\pm t\right) $ for some $\alpha \in \mathbb{R}^{\ast }.$ By using
the Lie group exponential maps, we shall deduce the classification of all
complete left-invariant affine structures on the oscillator group $O_{4}$ in
terms of simply transitive actions of subgroups of the affine group $%
Aff\left( \mathbb{R}^{4}\right) =GL\left( \mathbb{R}^{4}\right) \ltimes 
\mathbb{R}^{4}$ (see Theorem \ref{thm3}).

\smallskip

Throughout this paper, all vector spaces, Lie algebras, and left-symmetric
algebras are supposed to be over the filed $\mathbb{R}$, unless otherwise
specified. We shall also suppose that all Lie groups are connected and
simply connected.

\section{Extensions of Lie algebras}

Recall that a Lie algebra $\widetilde{\mathcal{G}}$ \ is an extension of the
Lie algebra $\mathcal{G}$ by the Lie algebra $\mathcal{A}$ if there exists a
short exact sequence of Lie algebras%
\begin{equation}
0\rightarrow \mathcal{A}\overset{i}{\rightarrow }\widetilde{\mathcal{G}}%
\overset{\pi }{\rightarrow }\mathcal{G}\rightarrow 0.  \label{seq1}
\end{equation}

In other words, if we identify the elements of $\mathcal{A}$ with their
images in $\widetilde{\mathcal{G}}$ via the injection $i,$ then $\mathcal{A}$
is an ideal in $\widetilde{\mathcal{G}}$ such that $\widetilde{\mathcal{G}}/%
\mathcal{A\cong }\mathcal{G}$ .

Two extensions $\widetilde{\mathcal{G}}_{1}$ and $\widetilde{\mathcal{G}}%
_{2} $ are called equivalent if there exists an isomorphism of Lie algebras $%
\varphi $ such that the following diagram commutes 
\begin{equation*}
\begin{array}{crclclc}
0\longrightarrow & \mathcal{A} & \overset{i_{1}}{\longrightarrow } & 
\widetilde{\mathcal{G}}_{1} & \overset{\pi _{1}}{\longrightarrow } & 
\mathcal{G} & \longrightarrow 0 \\ 
& id_{\mathcal{A}}\downarrow &  & \downarrow \varphi &  & \downarrow id_{%
\mathcal{G}} &  \\ 
0\longrightarrow & \mathcal{A} & \overset{i_{2}}{\longrightarrow } & 
\widetilde{\mathcal{G}}_{2} & \overset{\pi _{2}}{\longrightarrow } & 
\mathcal{G} & \longrightarrow 0%
\end{array}%
\end{equation*}

The notion of extensions of a Lie algebra $\mathcal{G}$ by an abelian Lie
algebra $\mathcal{A}$ is well known (see for instance, the books \cite{de
azcarraga} and \cite{jacobson}). In light of \cite{neeb}, we shall describe
here the notion of extension $\widetilde{\mathcal{G}}$ of a Lie algebra $%
\mathcal{G}$ by a Lie algebra $\mathcal{A}$ which is not necessarily abelian.

Suppose that a vector space extension $\widetilde{\mathcal{G}}$ of a Lie
algebra $\mathcal{G}$ by another Lie algebra $\mathcal{A}$ is known, and we
want to define a Lie structure on $\widetilde{\mathcal{G}}$ in terms of the
Lie structures of $\mathcal{G}$ and $\mathcal{A}$. Let $\sigma :\mathcal{G}%
\rightarrow \widetilde{\mathcal{G}}$ be a section, that is, a linear map
such that $\pi \circ \sigma =id.$ Then the linear map $\Psi :\left(
a,x\right) \mapsto i\left( a\right) +\sigma \left( x\right) $ from $\mathcal{%
A\oplus G}$ onto $\widetilde{\mathcal{G}}$ is an isomorphism of vector
spaces.

For $\left( a,x\right) $ and $\left( b,y\right) $ in $\mathcal{A\oplus G}$,
a commutator on $\widetilde{\mathcal{G}}$ must satisfy%
\begin{eqnarray}
\left[ i\left( a\right) +\sigma \left( x\right) ,i\left( b\right) +\sigma
\left( y\right) \right] &=&i\left( \left[ a,b\right] \right) +\left[ \sigma
\left( x\right) ,i\left( b\right) \right]  \label{eq0} \\
&&+\left[ i\left( a\right) ,\sigma \left( y\right) \right] +\left[ \sigma
\left( x\right) ,\sigma \left( y\right) \right]  \notag
\end{eqnarray}

Now we define a linear map $\phi :\mathcal{G}\rightarrow End\left( \mathcal{A%
}\right) $ by 
\begin{equation}
\phi \left( x\right) a=\left[ \sigma \left( x\right) ,i\left( a\right) %
\right]  \label{eq0.5}
\end{equation}

On the other hand, since 
\begin{equation*}
\pi \left( \left[ \sigma \left( x\right) ,\sigma \left( y\right) \right]
\right) =\pi \left( \sigma \left( \left[ x,y\right] \right) \right) ,
\end{equation*}%
it follows that there exists an alternating bilinear map $\omega :\mathcal{%
G\times G\rightarrow A}$ such that 
\begin{equation*}
\left[ \sigma \left( x\right) ,\sigma \left( y\right) \right] =\sigma \left[
x,y\right] +\omega \left( x,y\right) .
\end{equation*}

In summary, by means of the isomorphism above, $\widetilde{\mathcal{G}}\cong 
\mathcal{A\oplus G}$ and its elements may be denoted by $\left( a,x\right) $
with $a\in \mathcal{A}$ and $x$ is simply characterized by its coordinates
in $\mathcal{G}$. The commutator defined by (\ref{eq0}) is now given by 
\begin{equation}
\left[ \left( a,x\right) ,\left( b,y\right) \right] =\left( \left[ a,b\right]
+\phi \left( x\right) b-\phi \left( y\right) a+\omega \left( x,y\right) ,%
\left[ x,y\right] \right) ,  \label{eq1}
\end{equation}%
for all $\left( a,x\right) \in \widetilde{\mathcal{G}}\cong \mathcal{A\oplus
G}$.

\bigskip

Now, it is easy to see that this is actually a Lie bracket (i.e, it verifies
the Jacobi identity) if and only if the following three conditions are
satisfied

\begin{enumerate}
\item $\phi \left( x\right) \left[ b,c\right] =\left[ \phi \left( x\right)
b,c\right] +\left[ b,\phi \left( x\right) c\right] ,$

\item $\left[ \phi \left( x\right) ,\phi \left( y\right) \right] =\phi
\left( \left[ x,y\right] \right) +ad_{\omega \left( x,y\right) },$

\item $\omega \left( \left[ x,y\right] ,z\right) -\omega \left( x,\left[ y,z%
\right] \right) +\omega \left( y,\left[ x,z\right] \right) =\phi \left(
x\right) \omega \left( y,z\right) +\phi \left( y\right) \omega \left(
z,x\right) +\phi \left( z\right) \omega \left( x,y\right) .$
\end{enumerate}

\begin{remark}
\label{remark1}We see that the condition (1) above is equivalent to say that 
$\phi \left( x\right) $ is a derivation of $\mathcal{A}$. In other words, $%
\mathcal{G}$ is actually acting by derivations, that is, $\phi :\mathcal{G}%
\rightarrow Der\left( \mathcal{A}\right) .$ The condition (2) indicates
clearly that if $\mathcal{A}$ is supposed to be abelian, then $\mathcal{A}$
becomes a $\mathcal{G}$-module in a natural way, because in this case the
linear map $\phi :\mathcal{G}\rightarrow Der\left( \mathcal{A}\right) $
given by $\phi \left( x\right) a=\left[ \sigma \left( x\right) ,i\left(
a\right) \right] $ is well defined. The condition (3) is equivalent to the
fact that, if $\mathcal{A}$ is abelian, $\omega $ is a $2$-cocycle (i.e., $%
\delta _{\phi }\omega =0,$ where $\delta _{\phi }$ refers to the coboundary
operator corresponding to the action $\phi $).
\end{remark}

If now $\sigma ^{\prime }:\mathcal{G}\rightarrow \widetilde{\mathcal{G}}$ is
another section, then $\sigma ^{\prime }-\sigma =\tau $ for some linear map $%
\tau :\mathcal{G}\rightarrow \mathcal{A}$, and it follows that the
corresponding morphism and $2$-cocycle are, respectively, $\phi ^{\prime
}=\phi +ad\circ \tau $ and $\omega ^{\prime }=\omega +\delta _{\phi }\tau +%
\frac{1}{2}\left[ \tau ,\tau \right] ,$ where $ad$ stands here and below (if
there is no ambiguity) for the adjoint representation in $\mathcal{A}$, and
where $\left[ \tau ,\tau \right] $ has the following meaning : Given two
linear maps $\alpha ,\beta :\mathcal{G}\rightarrow \mathcal{A}$, we define $%
\left[ \alpha ,\beta \right] \left( x,y\right) =\left[ \alpha \left(
x\right) ,\beta \left( y\right) \right] -\left[ \alpha \left( y\right)
,\beta \left( x\right) \right] .$ In particular, we have $\frac{1}{2}\left[
\tau ,\tau \right] \left( x,y\right) =\left[ \tau \left( x\right) ,\tau
\left( y\right) \right] $. Note here that the Lie algebra $\mathcal{A}$ is
not necessarily abelian. Therefore, $\omega ^{\prime }-\omega $ is a $2$%
-coboundary if and only if $\left[ \tau \left( x\right) ,\tau \left(
y\right) \right] =0$ for all $x,y\in \mathcal{G}$. Equivalently, $\omega
^{\prime }-\omega $ is a $2$-coboundary if and only if $\tau $ has its range
in the center $Z\left( \mathcal{A}\right) $ of $\mathcal{A}$. In that case,
we get $\omega ^{\prime }-\omega =\delta _{\phi }\tau \in B_{\phi
}^{2}\left( \mathcal{G},Z\left( \mathcal{A}\right) \right) ,$ the group of $%
2 $-coboundaries for $\mathcal{G}\ $with values in $Z\left( \mathcal{A}%
\right) .$

\smallskip

To overcome all these difficulties, we proceed as follows. Let $C^{2}\left( 
\mathcal{G},\mathcal{A}\right) $ be the abelian group of all $2$-cochains,
i.e. alternating bilinear mappings $\mathcal{G}\times \mathcal{G}\rightarrow 
\mathcal{A}$. For a given $\phi :\mathcal{G}\rightarrow Der\left( \mathcal{A}%
\right) ,$ let $T_{\phi }\in C^{2}\left( \mathcal{G},\mathcal{A}\right) $ be
defined by 
\begin{equation*}
T_{\phi }\left( x,y\right) =\left[ \phi \left( x\right) ,\phi \left(
y\right) \right] -\phi \left( \left[ x,y\right] \right) ,\ \ \ \text{for all 
}x,y\in \mathcal{G}.
\end{equation*}

If there exists some $\omega \in C^{2}\left( \mathcal{G},\mathcal{A}\right) $
such that $T_{\phi }=ad\circ \omega ~$and $\delta _{\phi }\omega =0,$ then
the pair $\left( \phi ,\omega \right) $ is called a \textsl{factor system }%
for\textsl{\ }$\left( \mathcal{G},\mathcal{A}\right) .$ \textsl{\ }Let $%
Z^{2}\left( \mathcal{G},\mathcal{A}\right) $ be the set of all factor
systems for $\left( \mathcal{G},\mathcal{A}\right) .$ It is well known that
the equivalence classes of extensions of a Lie algebra $\mathcal{G}$ by a
Lie algebra $\mathcal{A}$ are in one-to-one correspondence with the elements
of the quotient space $Z^{2}\left( \mathcal{G},\mathcal{A}\right)
/C^{1}\left( \mathcal{G},\mathcal{A}\right) ,$ where $C^{1}\left( \mathcal{G}%
,\mathcal{A}\right) $ is the space of linear maps from $\mathcal{G}$ into $%
\mathcal{A}$ (see for instance \cite{neeb}, Theorem II.7). Note that if we
assume that $\mathcal{A}$ is abelian, then we meet the well known result
(see for instance \cite{chevalley}) stating that for a given action$\ \ \phi
:\mathcal{G}\rightarrow End\left( \mathcal{A}\right) ,$ the equivalence
classes of extensions of $\mathcal{G}$ by $\mathcal{A}$ are in one-to-one
correspondence with the elements of the second cohomology group 
\begin{equation*}
H_{\phi }^{2}\left( \mathcal{G},\mathcal{A}\right) =Z_{\phi }^{2}\left( 
\mathcal{G},\mathcal{A}\right) /B_{\phi }^{2}\left( \mathcal{G},\mathcal{A}%
\right) .
\end{equation*}

In the present paper, we shall be concerned with the special case where $%
\mathcal{A}$ is non-abelian and $\mathcal{G}$ is $\mathbb{R}$, and
henceforth the cocycle $\omega $ is identically zero.

\begin{remark}
It is worth noticing that the construction above is closely related to the
notion of $\mathcal{G}$-kernels (considered for Lie algebras firstly in \cite%
{mori}) . On $\left\{ \phi :\mathcal{G}\rightarrow Der\left( \mathcal{A}%
\right) :T_{\phi }=ad\circ \omega ,\ \text{for some }\omega \in C^{2}\left( 
\mathcal{G},\mathcal{A}\right) \right\} ,$ define an equivalence relation by 
$\phi \sim \phi ^{\prime }~$if and only if $\phi ^{\prime }=\phi +ad\circ
\tau ,\ $for some linear map $\tau :\mathcal{G}\rightarrow \mathcal{A}$. The
equivalence class $\left[ \phi \right] $ of $\phi $ is called a $\mathcal{G}$%
\emph{-kernel}. It turns out that if$\mathcal{A}$ is abelian, then a $%
\mathcal{G}$-kernel is nothing but a $\mathcal{G}$-module. By considering
the quotient morphism $\Pi :Der\left( \mathcal{A}\right) \rightarrow
Out\left( \mathcal{A}\right) =Der\left( \mathcal{A}\right) /ad_{\mathcal{A}%
}, $ and remarking that $\Pi \circ ad\circ \tau =0$ for any linear map $\tau
:\mathcal{G}\rightarrow \mathcal{A}$, we can naturally associate to each $%
\mathcal{G}$-kernel $\left[ \phi \right] $ the morphism $\phi =\Pi \circ %
\left[ \phi \right] :\mathcal{G}\rightarrow Out\left( \mathcal{A}\right) .$
\end{remark}

\section{Left-symmetric algebras}

The notion of a \emph{left-symmetric algebra} arises naturally in various
areas of mathematics and physics. It originally appeared in the works of
Vinberg \cite{vinberg} and Koszul \cite{koszul} concerning convex
homogeneous cones and bounded homogeneous domains, respectively. It also
appears, for instance, in connection with Yang-Baxter equation and
integrable hydrodynamic systems (cf. \cite{bordemann}, \cite{gol-sokolov}, 
\cite{kupershmidt}).

A left-symmetric algebra $\left( A,.\right) $ is a finite-dimensional
algebra $A$ in which the products, for all $x,y,z\in A,$ satisfy the identity

\begin{equation}
\left( xy\right) z-x\left( yz\right) =\left( yx\right) z-y\left( xz\right) ,
\label{eq2}
\end{equation}%
where here and frequently during this paper we simply write $xy$ instead of $%
x\cdot y$.

It is clear that an associative algebra is a left-symmetric algebra.
Actually, if $A$ is a left-symmetric algebra and $\left( x,y,z\right)
=\left( xy\right) z-x\left( yz\right) $ is the associator of $x,y,z$, then
we can see that (\ref{eq2}) is equivalent to $\left( x,y,z\right) =\left(
y,x,z\right) $ This means that the notion of a left-symmetric algebra is a
natural generalization of the notion of an associative algebra.

\smallskip

Now if $A$ is a left-symmetric algebra, then the commutator 
\begin{equation}
\left[ x,y\right] =xy-yx  \label{eq3}
\end{equation}%
defines a structure of Lie algebra on $A,$ called the \emph{associated Lie
algebra}. On the other hand, if $\mathcal{G}$ is a Lie algebra with a
left-symmetric product $\cdot $ satisfying 
\begin{equation*}
\left[ x,y\right] =x\cdot y-y\cdot x,
\end{equation*}%
then we say that the left-symmetric structure is \emph{compatible }with the
Lie structure on $\mathcal{G}$.

Suppose now we are given a Lie group $G$ with a left-invariant flat affine
connection $\nabla ,$ and define a product $\cdot $ on the Lie algebra $%
\mathcal{G}$ of $G$ by%
\begin{equation}
x\cdot y=\nabla _{x}y,  \label{eq4}
\end{equation}%
for all $x,y\in \mathcal{G}$. Then, the conditions on the connection $\nabla 
$ for being flat and torsion-free are now equivalent to the conditions (\ref%
{eq2}) and (\ref{eq3}), respectively.

Conversely, suppose that $G$ is a simply connected Lie group with Lie
algebra $\mathcal{G}$, and suppose that $\mathcal{G}$ is endowed with a
left-symmetric product $\cdot $ which is compatible with the Lie bracket of $%
\mathcal{G}$. We define an operator $\nabla $ on $\mathcal{G}$ according to
identity (\ref{eq4}), and then we extend it by left-translations to the
whole Lie group $G$. This clearly defines a left-invariant flat affine
structure on $G.$ In summary, for a given simply connected Lie group $G$
with Lie algebra $\mathcal{G}$, the left-invariant flat affine structures on 
$G$ are in one-to-one correspondence with the left-symmetric structures on $%
\mathcal{G}$ compatible with the Lie structure.

\smallskip

Let $A$ be a left-symmetric algebra, and let the left and right
multiplications $L_{x}$ and $R_{x}$ by the element $x$ be defined by $%
L_{x}y=x\cdot y$ and $R_{x}y=y\cdot x.$ We say that $A$ is \emph{complete}
if $R_{x}$ is a nilpotent operator, for all $x\in A.$ It turns out that, for
a given simply connected Lie group $G$ with Lie algebra $\mathcal{G}$, the
complete left-invariant flat affine structures on $G$ are in one-to-one
correspondence with the complete left-symmetric structures on $\mathcal{G}$
compatible with the Lie structure (see for example \cite{kim}). It is also
known that an $n$-dimensional simply connected Lie group admits a complete
left-invariant flat affine structure if and only if it acts simply
transitively on $\mathbb{R}^{n}$ by affine transformations (see \cite{kim}).
A simply connected Lie group which is acting simply transitively on $\mathbb{%
R}^{n}$ by affine transformations must be solvable according to \cite%
{auslander}, but it is worth noticeable that there exist solvable (even
nilpotent) Lie groups which do not admit affine structures (see \cite%
{benoist}).

\smallskip

We close this section by fixing some notations which we will use in what
follows. For a left-symmetric algebra $A$, we can easily check that the
subset 
\begin{equation}
T\left( A\right) =\left\{ x\in A:L_{x}=0\right\}  \label{T(A)}
\end{equation}%
is a two-sided ideal in $A.$ Geometrically, if $G$ is a Lie group which acts
simply transitively on $\mathbb{R}^{n}$ by affine transformations then $%
T\left( \mathcal{G}\right) $ corresponds to the set of translational
elements in $G,$ where $\mathcal{G}$ is endowed with the complete
left-symmetric product corresponding to the action of $G$ on $\mathbb{R}%
^{n}. $ It has been conjectured in \cite{auslander} that every nilpotent Lie
group $G$ which acts simply transitively on $\mathbb{R}^{n}$ by affine
transformations contains a translation which lies in the center of $G$, but
this conjecture turned out to be false (see \cite{fried}).

\subsection{Extensions of left-symmetric algebras}

We have discussed in the last section the problem of extension of a Lie
algebra by another Lie algebra. Similarly, we shall briefly discuss in this
section the problem of extension of a left-symmetric algebra by another
left-symmetric algebra. To our knowledge, the notion of extensions of
left-symmetric algebras has been considered for the first time in \cite{kim}%
, to which we refer for more details.

Suppose we are given a vector space $A$ as an extension of a left-symmetric
algebra $K$ by another left-symmetric algebra $E$ . We want to define a
left-symmetric structure on $A$ in terms of the left-symmetric structures
given on $K$ and $E.$ In other words, we want to define a left-symmetric
product on $A$ for which $E$ becomes a two-sided ideal in $A$ such that $%
A/E\cong K;$ or equivalently, 
\begin{equation*}
0\rightarrow E\rightarrow A\rightarrow K\rightarrow 0
\end{equation*}%
becomes a short exact sequence of left-symmetric algebras.

\begin{theorem}[\protect\cite{kim}]
\label{thm1}There exists a left-symmetric structure on $A$ extending a
left-symmetric algebra $K$ by a left-symmetric algebra $E$ if and only if
there exist two linear maps $\lambda ,$ $\rho :K\rightarrow End\left(
E\right) $ and a bilinear map $g:K\times K\rightarrow E$ such that, for all $%
x,y,z\in K$ and $a,b\in E,$ the following conditions are satisfied.

\begin{description}
\item[(i)] $\lambda _{x}\left( a\cdot b\right) =\lambda _{x}\left( a\right)
\cdot b+a\cdot \lambda _{x}\left( b\right) -\rho _{x}\left( a\right) \cdot
b, $

\item[(ii)] $\rho _{x}\left( \left[ a,b\right] \right) =a\cdot \rho
_{x}\left( b\right) -b\cdot \rho _{x}\left( a\right) ,$

\item[(iii)] $\left[ \lambda _{x},\lambda _{y}\right] =\lambda _{\left[ x,y%
\right] }+L_{g\left( x,y\right) -g\left( y,x\right) },$ where $L_{g\left(
x,y\right) -g\left( y,x\right) }$ denotes the left multiplication in $E$ by $%
g\left( x,y\right) -g\left( y,x\right) .$

\item[(iv)] $\left[ \lambda _{x},\rho _{y}\right] =\rho _{x\cdot y}-\rho
_{y}\circ \rho _{x}+R_{g\left( x,y\right) },$ where $R_{g\left( x,y\right) }$
denotes the right multiplication in $E$ by $g\left( x,y\right) .$

\item[(v)] $g\left( x,y\cdot z\right) -g\left( y,x\cdot z\right) +\lambda
_{x}\left( g\left( y,z\right) \right) -\lambda _{y}\left( g\left( x,z\right)
\right) -g\left( \left[ x,y\right] ,z\right) $

\item $-\rho _{z}\left( g\left( x,y\right) -g\left( y,x\right) \right) =0.$
\end{description}
\end{theorem}

If the conditions of Theorem \ref{thm1} are fulfilled, then the extended
left-symmetric product on $A\cong K\times E$ is given by%
\begin{equation}
\left( x,a\right) \cdot \left( y,b\right) =\left( x\cdot y,a\cdot b+\lambda
_{x}\left( b\right) +\rho _{y}\left( a\right) +g\left( x,y\right) \right) .
\label{eq5}
\end{equation}

It is remarkable that if the left-symmetric product of $E$ is trivial, then
the conditions of Theorem \ref{thm1} simplify to the following three
conditions:

\begin{description}
\item[(i)] $\left[ \lambda _{x},\lambda _{y}\right] =\lambda _{\left[ x,y%
\right] },$ i.e. $\lambda $ is a representation of Lie algebras,

\item[(ii)] $\left[ \lambda _{x},\rho _{y}\right] =\rho _{x\cdot y}-\rho
_{y}\circ \rho _{x}.$

\item[(iii)] $g\left( x,y\cdot z\right) -g\left( y,x\cdot z\right) +\lambda
_{x}\left( g\left( y,z\right) \right) -\lambda _{y}\left( g\left( x,z\right)
\right) -g\left( \left[ x,y\right] ,z\right) $

\item $-\rho _{z}\left( g\left( x,y\right) -g\left( y,x\right) \right) =0.$
\end{description}

\smallskip In this case, $E$ becomes a $K$-bimodule and the extended product
given in (\ref{eq5}) simplifies too.

Recall that if $K$ is a left-symmetric algebra and $V$ is a vector space,
then we say that $V$ is a $K$-bimodule if there exist two linear maps $%
\lambda ,$ $\rho :K\rightarrow End\left( V\right) $ which satisfy the
conditions (i) and (ii) stated above.

\smallskip

Let $K$ be a left-symmetric algebra, and let $V$ be a $K$-bimodule. Let $%
L^{p}\left( K,V\right) $ be the space of all $p$-linear maps from $K$ to $V,$
and define two coboundary operators $\delta _{1}:L^{1}\left( K,V\right)
\rightarrow L^{2}\left( K,V\right) $ and $\delta _{2}:L^{2}\left( K,V\right)
\rightarrow L^{3}\left( K,V\right) $ as follows : For a linear map $h\in
L^{1}\left( K,V\right) $ we set 
\begin{equation}
\delta _{1}h\left( x,y\right) =\rho _{y}\left( h\left( x\right) \right)
+\lambda _{x}\left( h\left( y\right) \right) -h\left( x\cdot y\right) ,
\label{delta1}
\end{equation}%
and for a bilinear map $g\in L^{2}\left( K,V\right) $ we set 
\begin{eqnarray}
\delta _{2}g\left( x,y,z\right) &=&g\left( x,y\cdot z\right) -g\left(
y,x\cdot z\right) +\lambda _{x}\left( g\left( y,z\right) \right) -\lambda
_{y}\left( g\left( x,z\right) \right)  \label{delta2} \\
&&-g\left( \left[ x,y\right] ,z\right) -\rho _{z}\left( g\left( x,y\right)
-g\left( y,x\right) \right) .  \notag
\end{eqnarray}

It is straightforward to check that $\delta _{2}\circ \delta _{1}=0.$
Therefore, if we set $Z_{\lambda ,\rho }^{2}\left( K,V\right) =\ker \delta
_{2}$ and $B_{\lambda ,\rho }^{2}\left( K,V\right) =\func{Im}\delta _{1}$,
we can define a notion of second cohomology for the actions $\lambda $ and $%
\rho $ by simply setting $H_{\lambda ,\rho }^{2}\left( K,V\right)
=Z_{\lambda ,\rho }^{2}\left( K,V\right) /B_{\lambda ,\rho }^{2}\left(
K,V\right) .$

As in the case of extensions of Lie algebras, we can prove that for given
linear maps $\lambda ,$ $\rho :K\rightarrow End\left( V\right) $, the
equivalence classes of extensions $0\rightarrow V\rightarrow A\rightarrow
K\rightarrow 0$ of $K$ by $V$ are in one-to-one correspondence with the
elements of the second cohomology group $H_{\lambda ,\rho }^{2}\left(
K,V\right) .$

\bigskip 

We close this subsection with the following lemma on completeness of
left-symmetric algebras (see Proposition 3.4 of \cite{changkimlee}). 

\begin{lemma}
\label{completelemma}Let $0\rightarrow E\rightarrow A\rightarrow
K\rightarrow 0$ be a short exact sequence of left-symmetric algebras. Then, $%
A$ is complete if and only if $E$ and $K$ are so.
\end{lemma}

\subsection{Central extensions of left-symmetric algebras}

The notion of central extensions known for Lie algebras may analogously be
defined for left-symmetric algebras. Let $A$ be a left-symmetric extension
of a left-symmetric algebra $K$ by another left-symmetric algebra $E$, and
let $\mathcal{G}$ be the Lie algebra associated to $A.$ Define \textsl{the
center} $C\left( A\right) $ of $A$ to be

\begin{equation}
C\left( A\right) =T\left( A\right) \cap Z\left( \mathcal{G}\right) =\left\{
x\in A:x\cdot y=y\cdot x=0,\ \ \ \text{for all }y\in A\right\} ,
\label{l.s. center}
\end{equation}%
where $Z\left( \mathcal{G}\right) $ is the center of the Lie algebra $%
\mathcal{G}$ and $T\left( A\right) $ is the two-sided ideal of $A$ defined
by (\ref{T(A)}).

\begin{definition}
The extension $0\rightarrow E\overset{i}{\rightarrow }A\overset{\pi }{%
\rightarrow }K\rightarrow 0$ of left-symmetric algebras is said to be 
\textsl{central} (resp. \textsl{exact}) if $i\left( E\right) \subseteq
C\left( A\right) $ (resp. $i\left( E\right) =C\left( A\right) $).\textsl{\ }
\end{definition}

\begin{remark}
It is not difficult to show that if the extension $0\rightarrow E\overset{i}{%
\rightarrow }A\overset{\pi }{\rightarrow }K\rightarrow 0$ is central, then
both the left-symmetric product and the $K$-bimodule on $E$ are trivial
(i.e., $a\cdot b=0$ for all $a,b\in E,$ and $\lambda =\rho =0$). It is also
easy to show that if $\left[ g\right] $ is the cohomology class associated
to this extension, and if 
\begin{equation*}
I_{\left[ g\right] }=\left\{ x\in K:x\cdot y=y\cdot x=0\ \text{and }g\left(
x,y\right) =g\left( y,x\right) =0,\ \text{for all }y\in K\right\} ,
\end{equation*}%
then the extension is exact if and only if $I_{\left[ g\right] }=0$ (see 
\cite{kim}). We note here that $I_{\left[ g\right] }$ is well defined
because any other element in $\left[ g\right] $ takes the form $g+\delta
_{1}h,$ with $\delta _{1}h\left( x,y\right) =-h\left( x\cdot y\right) .$
\end{remark}

Let now $K$ be a left-symmetric algebra, and $E$ a trivial $K$-bimodule.
Denote by $\left( A,\left[ g\right] \right) $ the central extension $%
0\rightarrow E\rightarrow A\rightarrow K\rightarrow 0$ corresponding to the
cohomology class $\left[ g\right] \in H^{2}\left( K,E\right) .$ Let $\left(
A,\left[ g\right] \right) $ and $\left( A^{\prime },\left[ g^{\prime }\right]
\right) $ be two central extensions of $K$ by $E,$ and let $\mu \in
Aut\left( E\right) =GL\left( E\right) $ and $\eta \in Aut\left( K\right) ,$
where $Aut\left( E\right) $ and $Aut\left( K\right) $ are the groups of
left-symmetric automorphisms of $E$ and $K,$ respectively. It is clear that,
if $h\in L^{1}\left( K,E\right) ,$ then the linear mapping $\psi
:A\rightarrow A^{\prime }$ defined by%
\begin{equation*}
\psi \left( x,a\right) =\left( \eta \left( x\right) ,\mu \left( a\right)
+h\left( x\right) \right)
\end{equation*}%
is an isomorphism provided $g^{\prime }\left( \eta \left( x\right) ,\eta
\left( y\right) \right) =\mu \left( g\left( x,y\right) \right) -\delta
_{1}h\left( x,y\right) $ for all $\left( x,y\right) \in K\times K,$ i.e. $%
\eta ^{\ast }\left[ g^{\prime }\right] =\mu _{\ast }\left[ g\right] .$

This allows us to define an action of the group $G=Aut\left( E\right) \times
Aut\left( K\right) $ on $H^{2}\left( K,E\right) $ by setting

\begin{equation}
\left( \mu ,\eta \right) .\left[ g\right] =\mu _{\ast }\eta ^{\ast }\left[ g%
\right] ,  \label{action1}
\end{equation}%
or equivalently, $\left( \mu ,\eta \right) .g\left( x,y\right) =\mu \left(
g\left( \eta \left( x\right) ,\eta \left( y\right) \right) \right) $ for all 
$x,y\in K.$

Denoting the set of all exact central extensions of $K$ by $E$ by%
\begin{equation*}
H_{ex}^{2}\left( K,E\right) =\left\{ \left[ g\right] \in H^{2}\left(
K,E\right) :I_{\left[ g\right] }=0\right\} ,
\end{equation*}%
and the orbit of $\left[ g\right] $ by $G_{\left[ g\right] },$ it turns out
that the following result is valid (see \cite{kim}).

\begin{proposition}
\label{prop3} Let $\left[ g\right] $ and $\left[ g^{\prime }\right] $ be two
classes in $H_{ex}^{2}\left( K,E\right) .$ Then, the central extensions $%
\left( A,\left[ g\right] \right) $ and $\left( A^{\prime },\left[ g^{\prime }%
\right] \right) $ are isomorphic if and only if $G_{\left[ g\right] }=G_{%
\left[ g^{\prime }\right] }.$ In other words, the classification of the
exact central extensions of $K$ by $E$ is, up to left-symmetric isomorphism,
the orbit space of $H_{ex}^{2}\left( K,E\right) $ under the natural action
of $G=Aut\left( E\right) \times Aut\left( K\right) .$
\end{proposition}

\subsection{Complexification of a real left-symmetric algebra}

Let $A$ be a real left-symmetric algebra of dimension $n,$ and let $A^{%
\mathbb{C}}$ denote the real vector space $A\oplus A.$ Let $J:A\oplus
A\rightarrow A\oplus A$ be the linear map on $A\oplus A$ defined by $J\left(
x,y\right) =\left( -y,x\right) .$

For $\alpha +i\beta \in \mathbb{C}$ and $x,x^{\prime },y,y^{\prime }\in A,$
we define%
\begin{equation}
\left( \alpha +i\beta \right) \left( x,y\right) =\left( \alpha x-\beta
y,\alpha y+\beta x\right)  \label{cmplx1}
\end{equation}%
and%
\begin{equation}
\left( x,y\right) \cdot \left( x^{\prime },y^{\prime }\right) =\left(
xx^{\prime }-yy^{\prime },xy^{\prime }+yx^{\prime }\right)  \label{cmplx2}
\end{equation}

We endow the set $A^{\mathbb{C}}$ with the componentwise addition,
multiplication by complex numbers defined by (\ref{cmplx1}), and the product
defined by (\ref{cmplx2}). It is then straightforward to verify that $A^{%
\mathbb{C}},$ when endowed with the product defined by (\ref{cmplx2}),
becomes a complex left-symmetric algebra called \textsl{the complexification 
}of $A.$

The left-symmetric algebra $A$ can be identified with the set of elements in 
$A^{\mathbb{C}}$ of the form $\left( x,0\right) ,$ where $x\in A.$ If $%
e_{1},\ldots ,e_{n}$ is a basis of $A,$ then the elements $\left(
e_{1},0\right) ,\ldots ,\left( e_{n},0\right) $ form a basis of the complex
vector space $A^{\mathbb{C}}.$ It follows that $\dim _{\mathbb{C}}\left( A^{%
\mathbb{C}}\right) =\dim _{\mathbb{R}}\left( A\right) .$

Since $A^{\mathbb{C}}$ is a left-symmetric algebra, we know that the
commutator%
\begin{equation*}
\left[ \left( x,y\right) ,\left( x^{\prime },y^{\prime }\right) \right]
=\left( x,y\right) \cdot \left( x^{\prime },y^{\prime }\right) -\left(
x^{\prime },y^{\prime }\right) \cdot \left( x,y\right)
\end{equation*}%
defines a Lie algebra $\mathcal{G}^{\mathbb{C}}$ on $A^{\mathbb{C}}.$

\medskip

Computing this commutator%
\begin{eqnarray*}
\left[ \left( x,y\right) ,\left( x^{\prime },y^{\prime }\right) \right]
&=&\left( x,y\right) \cdot \left( x^{\prime },y^{\prime }\right) -\left(
x^{\prime },y^{\prime }\right) \cdot \left( x,y\right) \\
&=&\left( \left[ x,x^{\prime }\right] -\left[ y,y^{\prime }\right] ,\left[
x,y^{\prime }\right] +\left[ y,x^{\prime }\right] \right) ,
\end{eqnarray*}%
we get the following

\begin{lemma}
\label{cmplxlemma}The complex Lie algebra $\mathcal{G}^{\mathbb{C}}$
associated to the complex left-symmetric algebra $A^{\mathbb{C}}$ is
isomorphic to the complexification of the Lie algebra $\mathcal{G}$
associated to the left-symmetric algebra $A.$
\end{lemma}

Therefore, if $e_{1},\ldots ,e_{n}$ is a basis of $A,$ then the elements $%
\left( e_{1},0\right) ,\ldots ,\left( e_{n},0\right) $ form a basis of $%
\mathcal{G}^{\mathbb{C}},$ and the structural constants of $\mathcal{G}^{%
\mathbb{C}}$ are real since they coincide with the structural constants of $%
\mathcal{G}$ in the basis $e_{1},\ldots ,e_{n}.$

\section{Left-symmetric structures on the oscillator algebra}

\bigskip Recall that the Heisenberg group $H_{3}$ is the 3-dimensional Lie
group diffeomorphic to $\mathbb{R\times C}$ with the group law 
\begin{equation*}
(v_{1},z_{1})\cdot (v_{2},z_{2})=(v_{1}+v_{2}+\frac{1}{2}\func{Im}(\overline{%
z_{1}}z_{2}),z_{1}+z_{2}),
\end{equation*}%
for all $v_{1},v_{2}\in \mathbb{R}$ and $z_{1},z_{2}\in \mathbb{C}$.

Let $\lambda >0,$ and let $G=\mathbb{Rn}H_{3}$ be equipped with the group
law 
\begin{equation*}
(t_{1},v_{1},z_{1})\cdot (t_{2},v_{2},z_{2})=(t_{1}+t_{2},v_{1}+v_{2}+\frac{1%
}{2}\func{Im}(\overline{z_{1}}z_{2}e^{i\lambda
t_{1}}),z_{1}+z_{2}e^{i\lambda t_{1}}),
\end{equation*}%
for all $t_{1},t_{2}\in \mathbb{R}$ and $(v_{1},z_{1}),(v_{2},z_{2})\in
H_{3}.$ This is a 4-dimensional Lie group with Lie algebra $\mathcal{G}$
having a basis $\{e_{1},e_{2},e_{3},e_{4}\}$ such that 
\begin{equation*}
\lbrack e_{1},e_{2}]=e_{3},\ [e_{4},e_{1}]=\lambda e_{2},\ \left[ e_{4},e_{2}%
\right] =-\lambda e_{1},
\end{equation*}%
and all the other brackets are zero.

It follows that the derived series is given by 
\begin{equation*}
\mathcal{D}^{1}\mathcal{G}=[\mathcal{G},\mathcal{G}]=span\{e_{1},e_{2},e_{3}%
\},\ \mathcal{D}^{2}\mathcal{G}=span\{e_{3}\},\ \mathcal{D}^{3}\mathcal{G}%
=\{0\},
\end{equation*}%
and therefore $\mathcal{G}$ is a (non-nilpotent) 3-step solvable Lie algebra.

\medskip

When $\lambda =1,$ $G$ is known as the \emph{oscillator group. }We shall
denote it by $O_{4}$, and we shall denote its Lie algebra by $\mathcal{O}%
_{4} $ and call it the \emph{oscillator algebra.}

\bigskip

Let $A_{4}$ be a complete real left-symmetric algebra whose associated Lie
algebra is $\mathcal{O}_{4}.$ We begin by proving the following proposition
which will be crucial to the classification of complete left-symmetric
structures on $\mathcal{O}_{4}.$

\begin{proposition}
$A_{4}$ is not simple (i.e., $A_{4}$ contains a proper two-sided ideal).
\end{proposition}

\begin{proof}
Assume to the contrary that $A_{4}$ is simple, and let $A_{4}^{\mathbb{C}}$
be its complexification. By \cite{kbm}, Lemma 2.10, it follows that $A_{4}^{%
\mathbb{C}}$ is either simple or a direct sum of two simple ideals having
the same dimension. If $A_{4}^{\mathbb{C}}$ is simple, then we can apply
Proposition 5.1 in \cite{burde} to deduce that, being simple and complete, $%
A_{4}^{\mathbb{C}}$ is necessarily isomorphic to the complex left-symmetric
algebra $B_{4}$ having a basis $\left\{ e_{1},e_{2},e_{3},e_{4}\right\} $
such that 
\begin{eqnarray*}
e_{1}\cdot e_{2} &=&e_{2}\cdot e_{1}=e_{4},~e_{2}\cdot e_{3}=2e_{1},~ \\
e_{3}\cdot e_{2} &=&e_{4}\cdot e_{1}=e_{1},~e_{4}\cdot
e_{2}=-e_{2},~e_{4}\cdot e_{3}=2e_{3},
\end{eqnarray*}%
and all other products are zero. It follows that the Lie algebra $\mathcal{G}%
_{4}$ associated to $B_{4}$ admits a basis $\left\{
e_{1},e_{2},e_{3},e_{4}\right\} $ such that%
\begin{equation*}
\left[ e_{2},e_{3}\right] =\left[ e_{4},e_{1}\right] =e_{1},~\left[
e_{2},e_{4}\right] =e_{2},~\left[ e_{3},e_{4}\right] =-2e_{3}.
\end{equation*}

This leads to a contradiction since, according to Lemma \ref{cmplxlemma}, $%
\mathcal{G}_{4}$ should be isomorphic to the complexification of the Lie
algebra $\mathcal{O}_{4},$ but this is obviously not the case. This
contradiction shows that $A_{4}^{\mathbb{C}}$ cannot be simple.

If $A_{4}^{\mathbb{C}}$ is a direct sum of two simple ideals having the same
dimension, say $A_{4}^{\mathbb{C}}=A_{1}\oplus A_{2},$ it follows that $\dim
A_{1}=\dim A_{2}=\frac{1}{2}\dim A_{4}^{\mathbb{C}}=2.$ In this case, by
Corollary 4.1 in \cite{burde}, $A_{1}$ and $A_{2}$ are both isomorphic to
the unique two-dimensional complex simple left-symmetric algebra having a
basis $B_{2}=\left\langle e_{1},e_{2}:e_{1}\cdot e_{1}=2e_{1},~e_{1}\cdot
e_{2}=e_{2},~e_{2}\cdot e_{2}=e_{1}\right\rangle .$ This is also a
contradiction, because $A_{1}$ and $A_{2}$ are complete however $B_{2}$ is
not complete. This contradiction shows that $A_{4}^{\mathbb{C}}$ cannot be
direct sum of two simple ideals. We deduce that $A_{4}$ is not simple, and
this completes the proof of the proposition.
\end{proof}

\bigskip

Before returning to the left-symmetric algebra $A_{4},$ we need to state the
following facts without proofs.

\begin{lemma}
\label{lemma0} Let $A$ be a left-symmetric algebra with Lie algebra $%
\mathcal{G}$, and $R$ a two-sided ideal in $A.$ Then, the Lie algebra $%
\mathcal{R}$ associated to $R$ is an ideal in $\mathcal{G}$.
\end{lemma}

\begin{lemma}
\label{lemma1} The oscillator algebra $\mathcal{O}_{4}$ contains only two
proper ideals which are $Z\left( \mathcal{O}_{4}\right) \cong \mathbb{R}$
and $\left[ \mathcal{O}_{4},\mathcal{O}_{4}\right] \cong \mathcal{H}_{3}.$
\end{lemma}

By the above proposition, $A_{4}$ is not simple and hence it has a proper
two-sided ideal $I,$ so we get a short exact sequence of complete
left-symmetric algebras 
\begin{equation}
0\rightarrow I\overset{i}{\rightarrow }A_{4}\overset{\pi }{\rightarrow }%
J\rightarrow 0.  \label{seq1.5}
\end{equation}

In fact, the completeness of $I$ and $J$ comes from that of $A_{4}$
according to Lemma \ref{completelemma}. If $\mathcal{I}$ is the Lie
subalgebra associated to $I$ then, by Lemma \ref{lemma0}, $\mathcal{I}$ is
an ideal in $\mathcal{O}_{4}$. From Lemma \ref{lemma1}, it follows that
there are two cases to consider according to whether $\mathcal{I}$ is
isomorphic to $\mathcal{H}_{3}$ or $\mathbb{R}$.

\bigskip

Next, we shall focus on the case where\textbf{\ }$\mathcal{I}$\textbf{\ }is
isomorphic to $\mathcal{H}_{3}\cong \left[ \mathcal{O}_{4},\mathcal{O}_{4}%
\right] .$ In this case,$\ $the short exact sequence (\ref{seq1.5}) becomes 
\begin{equation}
0\rightarrow I_{3}\overset{i}{\rightarrow }A_{4}\overset{\pi }{\rightarrow }%
I_{0}\rightarrow 0,  \label{seq2}
\end{equation}%
where $I_{3}$ is a complete $3$-dimensional left-symmetric algebra whose
underlying Lie algebra is $\mathcal{H}_{3},$ and $I_{0}$ is the trivial
one-dimensional real left-symmetric algebra $\left\{ e_{0}:e_{0}\cdot
e_{0}=0\right\} .$

It is not hard to prove the following proposition (compare \cite{friedgold},
Theorem 3.5).

\begin{proposition}
\label{prop1}Up to left-symmetric isomorphism, the complete left-symmetric
structures on the Heisenberg algebra $\mathcal{H}_{3}$ are classified as
follows: There is a basis $\left\{ e_{1},e_{2},e_{3}\right\} $ of $\mathcal{H%
}_{3}$ relative to which the left-symmetric product is given by one of the
following classes:

\begin{description}
\item[(i)] $e_{1}\cdot e_{1}=pe_{3},$ $e_{2}\cdot e_{2}=qe_{3},$ $e_{1}\cdot
e_{2}=\frac{1}{2}e_{3},$ $e_{2}\cdot e_{1}=-\frac{1}{2}e_{3},$ where $p,q\in 
\mathbb{R}.$

\item[(ii)] $e_{1}\cdot e_{2}=me_{3},$ $e_{2}\cdot e_{1}=\left( m-1\right)
e_{3},$ $e_{2}\cdot e_{2}=e_{1},$ where $m\in \mathbb{R}.$
\end{description}
\end{proposition}

\smallskip

\begin{remark}
It is noticeable that the left-symmetric products on $\mathcal{H}_{3}$
belonging to class (i) in Proposition \ref{prop1} are obtained by central
extensions (in the sense fixed in Subsection 3.1) of $\mathbb{R}^{2}$
endowed with some complete left-symmetric structure by $I_{0}.$ However, the
left-symmetric products on $A_{3}$ belonging to class (ii) are obtained by
central extensions of the non-abelian two-dimensional Lie algebra $\mathcal{G%
}_{2}$ endowed with its unique complete left-symmetric structure by $I_{0}.$
\end{remark}

Now we can return to the short exact sequence (\ref{seq2}). First, let $%
\sigma :I_{0}\rightarrow A_{4}$ be a section, and set $\sigma \left(
e_{0}\right) =x_{0}\in A_{4},$ and define two linear maps $\lambda ,$ $\rho
\in End\left( I_{3}\right) $ by putting $\lambda \left( y\right) =x_{0}\cdot
y$ and $\rho \left( y\right) =y\cdot x_{0},$ and put $e=x_{0}\cdot x_{0}$
(clearly $e\in I_{3}$).

Let $g:I_{0}\times I_{0}\rightarrow I_{3}$ be the bilinear map defined by $%
g\left( e_{0},e_{0}\right) =e.$ It is obvious, using the notation of
Subsection 3.1, to verify that $\delta _{2}g=0,$ i.e. $g\in Z_{\lambda ,\rho
}^{2}\left( I_{0},I_{3}\right) $.

\medskip

The extended left-symmetric product on $I_{3}\oplus I_{0}$ given by (\ref%
{eq5}) turns out to take the simplified form%
\begin{equation*}
\left( x,ae_{0}\right) \cdot \left( y,be_{0}\right) =\left( x\cdot
y+a\lambda \left( y\right) +b\rho \left( x\right) +abe,0\right) ,
\end{equation*}%
for all $x,y\in I_{3}$ and $a,b\in \mathbb{R}.$

\medskip

The conditions in Theorem \ref{thm1} can be simplified to the following
conditions:

\begin{eqnarray}
\lambda \left( x\cdot y\right) &=&\lambda \left( x\right) \cdot y+x\cdot
\lambda \left( y\right) -\rho \left( x\right) \cdot y  \label{eq6.1} \\
\rho \left( \left[ x,y\right] \right) &=&x\cdot \rho \left( y\right) -y\cdot
\rho \left( x\right)  \label{eq6.2} \\
\left[ \lambda ,\rho \right] +\rho ^{2} &=&R_{e}  \label{eq6.3}
\end{eqnarray}

\medskip

Let $\phi :\mathbb{R}\rightarrow End\left( \mathcal{H}_{3}\right) $ be the
linear map defined by formula (\ref{eq0.5}). As we mentioned in Remark \ref%
{remark1}, $\mathbb{R}$ acts on $\mathcal{H}_{3}$ by derivations, that is, $%
\phi :\mathbb{R}\rightarrow Der\left( \mathcal{H}_{3}\right) .$ In
particular, we deduce in view of (\ref{eq1}) that $\lambda =D+\rho $ for
some derivation $D$ of $\mathcal{H}_{3}.$ The derivations of $\mathcal{H}%
_{3} $ are given by the following lemma, whose proof is straightforward and
is therefore omitted.

\begin{lemma}
\label{lemma2}In a basis $\left\{ e_{1},e_{2},e_{3}\right\} $ of $\mathcal{H}%
_{3}$ satisfying $\left[ e_{1},e_{2}\right] =e_{3}$, a derivation $D$ of $%
\mathcal{H}_{3}$ takes the form%
\begin{equation*}
D=\left( 
\begin{array}{ccc}
a_{1} & b_{1} & 0 \\ 
a_{2} & b_{2} & 0 \\ 
a_{3} & b_{3} & a_{1}+b_{2}%
\end{array}%
\right) .
\end{equation*}
\end{lemma}

\bigskip

On the other hand, observe that $\left( x,ae_{0}\right) \in T\left(
A_{4}\right) $ if and only if $\left( x,ae_{0}\right) \cdot \left(
y,be_{0}\right) =\left( 0,0\right) $ for all $\left( y,be_{0}\right) \in
I_{3}\oplus I_{0}$, or equivalently, $x\cdot y+a\lambda \left( y\right)
+b\rho \left( x\right) +abe=0$ for all $\left( y,be_{0}\right) \in
I_{3}\oplus I_{0}.$ Since $y\ $and $b$ are arbitrary, we conclude that this
is also equivalent to say that $\left( L_{x}\right) _{\mid
_{A_{3}}}=-a\lambda $ and $\rho \left( x\right) =-ae.$ In particular, an
element $x\in I_{3}$ belongs to $T\left( A_{4}\right) $ if and only if $%
\left( L_{x}\right) _{\mid _{I_{3}}}=0$ and $\rho \left( x\right) =0,$ or
equivalently, 
\begin{equation}
I_{3}\cap T\left( A_{4}\right) =T\left( I_{3}\right) \cap \ker \rho .
\label{eq7}
\end{equation}

\bigskip

The following lemma will be crucial for the classification of complete
left-symmetric structures on $\mathcal{O}_{4}.$

\begin{lemma}
\label{lemma3}The center $C\left( A_{4}\right) =T\left( A_{4}\right) \cap
Z\left( \mathcal{O}_{4}\right) $ is non-trivial.
\end{lemma}

\begin{proof}
In view of Proposition \ref{prop1}, we have to consider two cases.

\medskip

\textbf{Case 1.} Assume that there is a basis $\left\{
e_{1},e_{2},e_{3}\right\} $ of $\mathcal{H}_{3}$ relative to which the
left-symmetric product of $I_{3}$ is given by : $e_{1}\cdot e_{1}=pe_{3},$ $%
e_{2}\cdot e_{2}=qe_{3},$ $e_{1}\cdot e_{2}=\frac{1}{2}e_{3},$ $e_{2}\cdot
e_{1}=-\frac{1}{2}e_{3},$ where $p,q\in \mathbb{R}.$

Substituting $x=e_{1}$ and $y=e_{2}$ into (\ref{eq6.2}), we find that the
operator $\rho $ takes the form%
\begin{equation*}
\rho =\left( 
\begin{array}{ccc}
\alpha _{1} & \beta _{1} & 0 \\ 
\alpha _{2} & \beta _{2} & 0 \\ 
\alpha _{3} & \beta _{3} & \gamma _{3}%
\end{array}%
\right) ,
\end{equation*}%
with $\gamma _{3}=p\beta _{1}-q\alpha _{2}+\frac{1}{2}\left( \alpha
_{1}+\beta _{2}\right) .$ Since $\lambda =D+\rho $ for some $D\in \mathcal{H}%
_{3}$, we use Lemma \ref{lemma2} to deduce that 
\begin{equation*}
\lambda =\left( 
\begin{array}{ccc}
\alpha _{1}+a_{1} & \beta _{1}+b_{1} & 0 \\ 
\alpha _{2}+a_{2} & \beta _{2}+b_{2} & 0 \\ 
\alpha _{3}+a_{3} & \beta _{3}+b_{3} & \gamma _{3}+a_{1}+b_{2}%
\end{array}%
\right) .
\end{equation*}

Since $\left( L_{e_{3}}\right) _{\mid _{I_{3}}}=0$ and $e\in I_{3}$, then (%
\ref{eq6.3}) when applied to $e_{3}$ gives%
\begin{equation*}
\gamma _{3}^{2}e_{3}=e_{3}\cdot e=0,
\end{equation*}%
from which we get $\gamma _{3}=0,$ i.e. $\rho \left( e_{3}\right) =0.$ It
follows from (\ref{eq7}) that $e_{3}\in T\left( A_{4}\right) .$ Since $%
Z\left( \mathcal{O}_{4}\right) =\mathbb{R}e_{3},$ we deduce that $C\left(
A_{4}\right) =T\left( A_{4}\right) \cap Z\left( \mathcal{O}_{4}\right) \neq
0,$ as required.

\medskip

\textbf{Case 2.} Assume now that there is a basis $\left\{
e_{1},e_{2},e_{3}\right\} $ of $\mathcal{H}_{3}$ relative to which the
left-symmetric product of $I_{3}$ is given by : $e_{1}\cdot e_{2}=me_{3},$ $%
e_{2}\cdot e_{1}=\left( m-1\right) e_{3},$ $e_{2}\cdot e_{2}=e_{1},$ where $%
m $ is a real number.

Substituting successively $x=e_{1},~y=e_{2}$ and $x=e_{2},~y=e_{3}$ into
equation (\ref{eq6.2}), we find that the operator $\rho $ takes the form%
\begin{equation}
\rho =\left( 
\begin{array}{ccc}
\alpha _{1} & \beta _{1} & -\alpha _{2} \\ 
\alpha _{2} & \beta _{2} & 0 \\ 
\alpha _{3} & \beta _{3} & m\beta _{2}-\left( m-1\right) \alpha _{1}%
\end{array}%
\right) ,  \label{eq8}
\end{equation}%
with~$\left( m-1\right) \alpha _{2}=0.$

We claim that\textbf{\ }$\alpha _{2}=0.$ To prove this, let us assume to the
contrary that $\alpha _{2}\neq 0.$ It follows that $m=1,$ and therefore%
\begin{eqnarray*}
\rho \left( e_{3}\right) &=&-\alpha _{2}e_{1}+\beta _{2}e_{3} \\
\rho ^{2}\left( e_{3}\right) &=&-\alpha _{2}\left( \alpha _{1}+\beta
_{2}\right) e_{1}-\alpha _{2}^{2}e_{2}+\left( \beta _{2}^{2}-\alpha
_{2}\alpha _{3}\right) e_{3}
\end{eqnarray*}

Since $\alpha _{2}\neq 0,$ we deduce that $e_{3},$ $\rho \left( e_{3}\right)
,$ $\rho ^{2}\left( e_{3}\right) $ form a basis of $I_{3}.$ Since $\rho $ is
nilpotent (by completeness of the left-symmetric structure), it follows that 
$\rho ^{3}\left( e_{3}\right) =0.$ In other words, $\rho $ has the form%
\begin{equation*}
\rho =\left( 
\begin{array}{ccc}
0 & 0 & 1 \\ 
-1 & 0 & 0 \\ 
0 & 0 & 0%
\end{array}%
\right) ,
\end{equation*}%
with respect to the basis $e_{1}^{\prime }=-\rho \left( e_{3}\right) ,$ $%
e_{2}^{\prime }=\rho ^{2}\left( e_{3}\right) ,$ $e_{3}^{\prime }=-e_{3}.$

Using the fact that $\alpha _{1}+2\beta _{2}=0$ which follows from the
identity $\rho ^{3}\left( e_{3}\right) =0,$ we see that $e_{1}^{\prime
}\cdot e_{2}^{\prime }=\alpha _{2}^{3}e_{3}^{\prime },$ $e_{2}^{\prime
}\cdot e_{2}^{\prime }=\alpha _{2}^{3}e_{1}^{\prime },$ and all other
products are zero.

For simplicity, assume without loss of generality that $\alpha _{2}=1.$
Since $\lambda =D+\rho $ for some $D\in \mathcal{H}_{3}$, Lemma \ref{lemma2}
tells us that, with respect to the basis $e_{1}^{\prime },e_{2}^{\prime },$ $%
e_{3}^{\prime }$, the operator $\lambda $ takes the form 
\begin{equation*}
\lambda =\left( 
\begin{array}{ccc}
a_{1} & b_{1} & 1 \\ 
a_{2}-1 & b_{2} & 0 \\ 
a_{3} & b_{3} & a_{1}+b_{2}%
\end{array}%
\right) .
\end{equation*}

Applying formula (\ref{eq6.3}) to $e_{3}^{\prime }$ and recalling that $%
e_{3}^{\prime }\cdot e=0$ since $\mathbf{e}\in I_{3},$ we deduce that $%
a_{2}=1$ and $b_{2}=a_{3}=0.$ Then, substituting $x=y=e_{2}^{\prime }$ into
equation (\ref{eq6.1}), we get $a_{1}=b_{1}=0.$ Thus, the form of $\lambda $
reduces to 
\begin{equation*}
\lambda =\left( 
\begin{array}{ccc}
0 & 0 & 1 \\ 
0 & 0 & 0 \\ 
0 & b_{3} & 0%
\end{array}%
\right) .
\end{equation*}

Now, by setting $\mathbf{e}=ae_{1}+be_{2}+ce_{3}$ and applying (\ref{eq6.3})
to $e_{1},$ we get that $b_{3}=-b.$ By using (\ref{eq5a}), we deduce that
the nonzero left-symmetric products are 
\begin{eqnarray*}
e_{1}^{\prime }\cdot e_{2}^{\prime } &=&e_{3}^{\prime },\ \ e_{2}^{\prime
}\cdot e_{2}^{\prime }=e_{1}^{\prime }, \\
e_{1}^{\prime }\cdot e_{4}^{\prime } &=&-e_{2}^{\prime },\ \ e_{4}^{\prime
}\cdot e_{2}^{\prime }=-be_{3}^{\prime } \\
e_{3}^{\prime }\cdot e_{4}^{\prime } &=&e_{4}^{\prime }\cdot e_{3}^{\prime
}=e_{1}^{\prime },\ \ e_{4}^{\prime }\cdot e_{4}^{\prime }=\mathbf{e}.
\end{eqnarray*}

This implies, in particular, that $\dim \left[ \mathcal{O}_{4},\mathcal{O}%
_{4}\right] =\dim \left[ A_{4},A_{4}\right] =2,$ a contradiction. It follows
that $\alpha _{2}=0,$ as desired.\medskip

We now return to (\ref{eq8}). Since $\alpha _{2}=0,$ we have 
\begin{equation*}
\rho =\left( 
\begin{array}{ccc}
\alpha _{1} & \beta _{1} & 0 \\ 
0 & \beta _{2} & 0 \\ 
\alpha _{3} & \beta _{3} & m\beta _{2}-\left( m-1\right) \alpha _{1}%
\end{array}%
\right) ,
\end{equation*}%
and since $\lambda =D+\rho $ for some $D\in \mathcal{H}_{3}$ then, in view
of Lemma \ref{lemma2}, the operator $\lambda $ takes the form 
\begin{equation*}
\lambda =\left( 
\begin{array}{ccc}
\alpha _{1}+a_{1} & \beta _{1}+b_{1} & 0 \\ 
a_{2} & \beta _{2}+b_{2} & 0 \\ 
\alpha _{3}+a_{3} & \beta _{3}+b_{3} & a_{1}+b_{2}+m\beta _{2}-\left(
m-1\right) \alpha _{1}%
\end{array}%
\right) .
\end{equation*}

Once again, by applying (\ref{eq6.3}) to $e_{3}$ and recalling that $%
e_{3}\cdot e=0$ since $\mathbf{e}\in I_{3},$ we deduce that $\left( m\beta
_{2}-\left( m-1\right) \alpha _{1}\right) ^{2}=0,$ thereby showing that $%
\rho \left( e_{3}\right) =0.$ Now, in view of (\ref{eq7}) we get $e_{3}\in
T\left( A_{4}\right) ,$ and since $Z\left( \mathcal{O}_{4}\right) =\mathbb{R}%
e_{3}$ we deduce that $C\left( A_{4}\right) =T\left( A_{4}\right) \cap
Z\left( \mathcal{O}_{4}\right) \neq 0,$ as desired. This completes the proof
of the lemma.
\end{proof}

\section{Classification}

We know from Section 4 that $A_{4}$\ has a proper two-sided ideal $I$\ which
is isomorphic to either the trivial one-dimensional real left-symmetric
algebra $I_{0}=\left\{ e_{0}:e_{0}\cdot e_{0}=0\right\} $ or a $3$%
-dimensional left-symmetric algebra $I_{3}$\ (as described in Proposition %
\ref{prop1}) whose associated Lie algebra is the Heisenberg algebra $%
\mathcal{H}_{3}$.

In case where $I\cong I_{3}$, we know by Lemma \ref{lemma3} that $C\left(
A_{4}\right) \neq \left\{ 0\right\} .$\ Since, in our situation, $\dim
Z\left( \mathcal{O}_{4}\right) =1,$\ it follows that $C\left( A_{4}\right)
\cong I_{0},$ so that we have a central short exact sequence of
left-symmetric algebras of the form%
\begin{equation}
0\rightarrow I_{0}\rightarrow A_{4}\rightarrow I_{3}\rightarrow 0.
\label{seq3}
\end{equation}

In general, one has that the center of a left-symmetric algebra is a part of
the center of the associated Lie algebra, and therefore the following lemma
is proved.

\begin{lemma}
\label{lemma5}The Lie algebra associated to $I_{3}$ is isomorphic to the Lie
algebra $\mathcal{E}\left( 2\right) $ of the group of Euclidean motions of
the plane.
\end{lemma}

Recall that $\mathcal{E}\left( 2\right) $ is solvable non-nilpotent and has
a basis$\ \left\{ e_{1},e_{2},e_{3}\right\} $ which satisfies $\left[
e_{1},e_{2}\right] =e_{3}$ and $\left[ e_{1},e_{3}\right] =-e_{2}$.

\bigskip

In the case where $I\cong I_{0}$, we know by Lemma \ref{lemma0} that the
associated Lie algebra is $\mathcal{I}\cong \mathbb{R}.$ Since, by Lemma \ref%
{lemma1}, $\mathcal{O}_{4}$ has only two proper ideals which are $Z\left( 
\mathcal{O}_{4}\right) \cong \mathbb{R}$ and $\left[ \mathcal{O}_{4},%
\mathcal{O}_{4}\right] \cong \mathcal{H}_{3},$ it follows that $\mathcal{I}%
\cong \mathbb{R}$ coincides with the center $Z\left( \mathcal{O}_{4}\right)
. $ We deduce from this that, if $\mathcal{J}$ denotes the Lie algebra of
the left-symmetric algebra $J$ in the short exact sequence (\ref{seq1.5}),
then $\mathcal{J}$ is isomorphic to $\mathcal{E}\left( 2\right) .$
Therefore, we have a short sequence of left-symmetric algebras which looks
like (\ref{seq3}) except that it would not necessarily be central. But, as
we will see a little later, this is necessarily a central extension (i.e., $%
I\cong C\left( A_{4}\right) \cong I_{0}$).

\bigskip

To summarize, each complete left-symmetric structure on $\mathcal{O}_{4}$\
may be obtained by extension of a complete $3$-dimensional left-symmetric
algebra $A_{3}$ whose associated Lie algebra is $\mathcal{E}\left( 2\right) $
by $I_{0}.$

Next, we shall determine all the complete left-symmetric structures on $%
\mathcal{E}\left( 2\right) .$ These are described by the following lemma
that we state without proof (see \cite{friedgold}, Theorem 4.1).

\begin{lemma}
\label{lemma6}Up to left-symmetric isomorphism, any complete left-symmetric
structure on $\mathcal{E}\left( 2\right) $ is isomorphic to the following
one which is given in a basis $\left\{ e_{1},e_{2},e_{3}\right\} $ of $%
\mathcal{E}\left( 2\right) $ by the relations 
\begin{equation*}
e_{1}\cdot e_{2}=e_{3},~e_{1}\cdot e_{3}=-e_{2},~e_{2}\cdot e_{2}=e_{3}\cdot
e_{3}=\varepsilon e_{1}.
\end{equation*}

There are exactly two non-isomorphic conjugacy classes according to whether $%
\varepsilon =0$ or $\varepsilon \neq 0.$
\end{lemma}

From now on, $A_{3}$ will denote the vector space $\mathcal{E}\left(
2\right) $ endowed with one of the complete left-symmetric structures
described in Lemma \ref{lemma6}. The extended Lie bracket on $\mathcal{E}%
\left( 2\right) \oplus \mathbb{R}$ is given by 
\begin{equation}
\left[ \left( x,a\right) ,\left( y,b\right) \right] =\left( \left[ x,y\right]
,\omega \left( x,y\right) \right) ,  \label{eq10}
\end{equation}%
with $\omega \in Z^{2}\left( \mathcal{E}\left( 2\right) ,\mathbb{R}\right) .$
The extended left-symmetric product on $A_{3}\oplus I_{0}$ is given by 
\begin{equation}
\left( x,ae_{0}\right) \cdot \left( y,be_{0}\right) =\left( x\cdot
y,b\lambda _{x}\left( e_{0}\right) +a\rho _{y}\left( e_{0}\right) +g\left(
x,y\right) \right) ,  \label{eq11}
\end{equation}%
with $\lambda ,$ $\rho :A_{3}\rightarrow End\left( I_{0}\right) $ and $g\in
Z_{\lambda ,\rho }^{2}\left( A_{3},I_{0}\right) .$

\bigskip

As we have noticed in Section 3, $I_{0}$ is an $A_{3}$-bimodule, or
equivalently, the conditions in Theorem \ref{thm1} simplify to the following
conditions:

\begin{description}
\item[(i)] $\lambda _{\left[ x,y\right] }=0.$

\item[(ii)] $\rho _{x\cdot y}=\rho _{y}\circ \rho _{x}.$

\item[(iii)] $g\left( x,y\cdot z\right) -g\left( y,x\cdot z\right) +\lambda
_{x}\left( g\left( y,z\right) \right) -\lambda _{y}\left( g\left( x,z\right)
\right) -g\left( \left[ x,y\right] ,z\right) $

\item $-\rho _{z}\left( g\left( x,y\right) -g\left( y,x\right) \right) =0.$
\end{description}

\smallskip

By using (\ref{eq10}) and (\ref{eq11}), we deduce from%
\begin{equation*}
\left[ \left( x,a\right) ,\left( y,b\right) \right] =\left( x,ae_{0}\right)
\cdot \left( y,be_{0}\right) -\left( y,be_{0}\right) \cdot \left(
x,ae_{0}\right) ,
\end{equation*}%
that 
\begin{equation}
\omega \left( x,y\right) =g\left( x,y\right) -g\left( y,x\right) \ \ \text{%
and }\lambda =\rho .  \label{eq12}
\end{equation}

\bigskip

By applying identity (ii) above to $e_{i}\cdot e_{i},$ $1\leq i\leq 3,$ we
deduce that $\rho =0,$ and a fortiori $\lambda =0.$ In other words, the
extension $A_{4}$ is always central (i.e., $I\cong C\left( A_{4}\right) $
even in the case where $\mathcal{I}\cong \mathbb{R}$).

\smallskip

In fact, we have

\begin{claim}
\label{claim2}The extension $0\rightarrow I_{0}\rightarrow A_{4}\rightarrow
A_{3}\rightarrow 0$ is exact.
\end{claim}

\begin{proof}
We recall from Subsection 3.1 that the extension given by the short sequence
(\ref{seq3}) is exact, i.e. $i\left( I_{0}\right) =C\left( A_{4}\right) $,
if and only if $I_{\left[ g\right] }=0,$ where%
\begin{equation*}
I_{\left[ g\right] }=\left\{ x\in A_{3}:x\cdot y=y\cdot x=0\ \text{and }%
g\left( x,y\right) =g\left( y,x\right) =0,\ \text{for all }y\in
A_{3}\right\} .
\end{equation*}%
To show that $I_{\left[ g\right] }=0$, let $x$ be an arbitrary element in $%
I_{\left[ g\right] }$, and put $x=ae_{1}+be_{2}+ce_{3}\in I_{\left[ g\right]
}.$ Now, by computing all the products $x\cdot e_{i}=e_{i}\cdot x=0,$ $1\leq
i\leq 3,$ we easily deduce that $x=0.$
\end{proof}

\bigskip

Our aim is to classify the complete left-symmetric structures on $\mathcal{O}%
_{4},$ up to left-symmetric isomorphisms. By Proposition \ref{prop3}, the
classification of the exact central extensions of $A_{3}$ by $I_{0}$ is, up
to left-symmetric isomorphism, the orbit space of $H_{ex}^{2}\left(
A_{3},I_{0}\right) $ under the natural action of $G=Aut\left( I_{0}\right)
\times Aut\left( A_{3}\right) .$ Accordingly, we must compute $%
H_{ex}^{2}\left( A_{3},I_{0}\right) .$ Since $I_{0}$ is a trivial $A_{3}$%
-bimodule, we see first from formulae (\ref{delta1}) and (\ref{delta2}) in
Section 3 that the coboundary operator $\delta $ simplifies as follows: 
\begin{eqnarray*}
\delta _{1}h\left( x,y\right)  &=&-h\left( x\cdot y\right) ,\  \\
\delta _{2}g\left( x,y,z\right)  &=&g\left( x,y\cdot z\right) -g\left(
y,x\cdot z\right) -g\left( \left[ x,y\right] ,z\right) ,
\end{eqnarray*}%
where $h\in L^{1}\left( A_{3},I_{0}\right) $ and $g\in L^{2}\left(
A_{3},I_{0}\right) .$

\bigskip

In view of Lemma \ref{lemma6}, there are two cases to be considered.

\smallskip

\smallskip

\bigskip \textbf{Case 1.} $A_{3}=\left\langle e_{1},e_{2},e_{3}:e_{1}\cdot
e_{2}=e_{3},~e_{1}\cdot e_{3}=-e_{2}\right\rangle .$

\smallskip

In this case, using the first formula above for $\delta _{1}$, we get%
\begin{equation*}
\delta _{1}h=\left( 
\begin{array}{ccc}
0 & h_{12} & h_{13} \\ 
0 & 0 & 0 \\ 
0 & 0 & 0%
\end{array}%
\right) ,
\end{equation*}%
where $h_{12}=-h\left( e_{3}\right) $ and $h_{13}=h\left( e_{2}\right) .$
Similarly, using the second formula above for $\delta _{2}$, we verify
easily that if $g$ is a cocycle (i.e. $\delta _{2}g=0$) and $g_{ij}=g\left(
e_{i},e_{j}\right) $, then 
\begin{equation*}
g=\left( 
\begin{array}{ccc}
g_{11} & g_{12} & g_{13} \\ 
0 & g_{22} & g_{23} \\ 
0 & -g_{23} & g_{22}%
\end{array}%
\right) ,
\end{equation*}%
that is, $g_{21}=g_{31}=0,$ $g_{32}=-g_{23},$ and $g_{33}=g_{22}.$ We deduce
that, in the basis above, the class $\left[ g\right] \in H^{2}\left( A_{3},%
\mathbb{R}\right) $ of a cocycle $g$ may be represented by a matrix of the
simplified form

\begin{equation*}
g=\left( 
\begin{array}{ccc}
\alpha & 0 & 0 \\ 
0 & \beta & \gamma \\ 
0 & -\gamma & \beta%
\end{array}%
\right) .
\end{equation*}

We can now determine the extended left-symmetric structure on $A_{4}.$ By
setting $\widetilde{e}_{i}=\left( e_{i},0\right) $, $1\leq i\leq 3,$ and $%
\widetilde{e}_{4}=\left( 0,1\right) ,$ and using formula (\ref{eq11}) which
(since $\lambda =\rho =0$) reduces to 
\begin{equation}
\left( x,ae_{0}\right) \cdot \left( y,be_{0}\right) =\left( x\cdot y,g\left(
x,y\right) \right) ,  \label{eq11bis}
\end{equation}%
we obtain 
\begin{eqnarray}
\widetilde{e}_{1}\cdot \widetilde{e}_{1} &=&\alpha \widetilde{e}_{4},\ 
\widetilde{e}_{2}\cdot \widetilde{e}_{2}=\widetilde{e}_{3}\cdot \widetilde{e}%
_{3}=\beta \widetilde{e}_{4}\   \notag \\
\widetilde{e}_{1}\cdot \widetilde{e}_{2} &=&\widetilde{e}_{3},\ \ \widetilde{%
e}_{1}\cdot \widetilde{e}_{3}=-\widetilde{e}_{2},  \label{eq16} \\
\widetilde{e}_{2}\cdot \widetilde{e}_{3} &=&\gamma \widetilde{e}_{4},\ \ 
\widetilde{e}_{3}\cdot \widetilde{e}_{2}=-\gamma \widetilde{e}_{4},  \notag
\end{eqnarray}%
and all the other products are zero. We observe here that we should have $%
\gamma \neq 0,$ given that the underlying Lie algebra is $\mathcal{O}_{4}.$
We denote by $A_{4}\left( \alpha ,\beta ,\gamma \right) $ the Lie algebra $%
\mathcal{O}_{4}$ endowed with the above complete left-symmetric
product.\bigskip 

Let now $A_{4}\left( \alpha ,\beta ,\gamma \right) $ and $A_{4}\left( \alpha
^{\prime },\beta ^{\prime },\gamma ^{\prime }\right) $ be two arbitrary
left-symmetric structures on $\mathcal{O}_{4}$ given as above, and let $%
\left[ g\right] $ and $\left[ g^{\prime }\right] $ be the corresponding
classes in $H_{ex}^{2}\left( A_{3},I_{0}\right) $. By Proposition \ref{prop3}%
, we know that $A_{4}\left( \alpha ,\beta ,\gamma \right) $ is isomorphic to 
$A_{4}\left( \alpha ^{\prime },\beta ^{\prime },\gamma ^{\prime }\right) $
if and only if the exists $\left( \mu ,\eta \right) \in Aut\left(
I_{0}\right) \times Aut\left( A_{3}\right) $ such that for all $x,y\in A_{3},
$ we have 
\begin{equation}
g^{\prime }\left( x,y\right) =\mu \left( g\left( \eta \left( x\right) ,\eta
\left( y\right) \right) \right) .  \label{action2}
\end{equation}

We shall first determine $Aut\left( I_{0}\right) \times Aut\left(
A_{3}\right) .$ We have $Aut\left( I_{0}\right) \cong \mathbb{R}^{\ast },$
and it is easy too to determine $Aut\left( A_{3}\right) .$ Indeed, recall
that the unique left-symmetric structure of $A_{3}$ is given by $e_{1}\cdot
e_{2}=e_{3},\ \ e_{1}\cdot e_{3}=-e_{2},$ and let $\eta \in Aut\left(
A_{3}\right) $ be given, in the basis $\left\{ e_{1},e_{2},e_{3}\right\} ,$
by 
\begin{equation*}
\eta =\left( 
\begin{array}{ccc}
a_{1} & b_{1} & c_{1} \\ 
a_{2} & b_{2} & c_{2} \\ 
a_{3} & b_{3} & c_{3}%
\end{array}%
\right) .
\end{equation*}

From the identity $\eta \left( e_{3}\right) =\eta \left( e_{1}\cdot
e_{2}\right) =\eta \left( e_{1}\right) \cdot \eta \left( e_{2}\right) ,$ we
get that $c_{1}=0,$ $c_{2}=-a_{1}b_{3},$ and $c_{3}=a_{1}b_{2}.$ From the
identity $-\eta \left( e_{2}\right) =\eta \left( e_{1}\cdot e_{3}\right)
=\eta \left( e_{1}\right) \cdot \eta \left( e_{3}\right) $ we get that $%
b_{1}=0,$ $b_{2}=a_{1}c_{3},$ and $b_{3}=-a_{1}c_{2}.$ Since $\det \eta \neq
0,$ we deduce that $a_{1}=\pm 1.$ It follows, by setting $\varepsilon =\pm
1, $ that $b_{3}=-\varepsilon c_{2}$ and $c_{3}=\varepsilon b_{2}.$ From the
identity $\eta \left( e_{1}\right) \cdot \eta \left( e_{1}\right) =\eta
\left( e_{1}\cdot e_{1}\right) =0,$ we obtain that $a_{2}=a_{3}=0.$
Therefore, $\eta $ takes the form%
\begin{equation*}
\eta =\left( 
\begin{array}{ccc}
\varepsilon & 0 & 0 \\ 
0 & b_{2} & c_{2} \\ 
0 & -\varepsilon c_{2} & \varepsilon b_{2}%
\end{array}%
\right) ,
\end{equation*}%
with $b_{2}^{2}+c_{2}^{2}\neq 0.$

We shall now apply formula (\ref{action2}). For this we recall first that,
in the basis above, the classes $\left[ g\right] $ and $\left[ g^{\prime }%
\right] $ corresponding, respectively, to $A_{4}\left( \alpha ,\beta ,\gamma
\right) $ and $A_{4}\left( \alpha ^{\prime },\beta ^{\prime },\gamma
^{\prime }\right) $ have, respectively, the forms 
\begin{equation*}
g=\left( 
\begin{array}{ccc}
\alpha & 0 & 0 \\ 
0 & \beta & \gamma \\ 
0 & -\gamma & \beta%
\end{array}%
\right) ~~~\text{and }~~g^{\prime }=\left( 
\begin{array}{ccc}
\alpha ^{\prime } & 0 & 0 \\ 
0 & \beta ^{\prime } & \gamma ^{\prime } \\ 
0 & -\gamma ^{\prime } & \beta ^{\prime }%
\end{array}%
\right) .
\end{equation*}

From $g^{\prime }\left( e_{1},e_{1}\right) =\mu g\left( \eta \left(
e_{1}\right) ,\eta \left( e_{1}\right) \right) ,$ we get 
\begin{equation}
\alpha ^{\prime }=\mu \alpha ,  \label{eq13}
\end{equation}%
and from $g^{\prime }\left( e_{2},e_{2}\right) =\mu g\left( \eta \left(
e_{2}\right) ,\eta \left( e_{2}\right) \right) $, we get 
\begin{equation}
\beta ^{\prime }=\mu \left( b_{2}^{2}+c_{2}^{2}\right) \beta .  \label{eq14}
\end{equation}

Similarly, from $g^{\prime }\left( e_{2},e_{3}\right) =\mu g\left( \eta
\left( e_{2}\right) ,\eta \left( e_{3}\right) \right) $ we get 
\begin{equation}
\gamma ^{\prime }=\mu \varepsilon \left( b_{2}^{2}+c_{2}^{2}\right) \gamma .
\label{eq15}
\end{equation}

Recall here that $\mu \neq 0$, $\gamma \neq 0,$ and $b_{2}^{2}+c_{2}^{2}\neq
0.$

\begin{claim}
\label{claim}\bigskip Each $A_{4}\left( \alpha ,\beta ,\gamma \right) $ is
isomorphic to some $A_{4}\left( \alpha ^{\prime },\beta ^{\prime },1\right)
. $ Precisely, $A_{4}\left( \alpha ,\beta ,\gamma \right) $ is isomorphic to 
$A_{4}\left( \varepsilon \frac{\alpha }{\gamma },\varepsilon \frac{\beta }{%
\gamma },1\right) .$
\end{claim}

\begin{proof}
By (\ref{eq13}), (\ref{eq14}), and (\ref{eq15}), $A_{4}\left( \alpha ,\beta
,\gamma \right) $ is isomorphic to $A_{4}\left( \alpha ^{\prime },\beta
^{\prime },1\right) $ if and only if there exists $\mu \in \mathbb{R}^{\ast
} $ and $b,c\in \mathbb{R}$, with $b^{2}+c^{2}\neq 0,$ such that 
\begin{eqnarray*}
\alpha ^{\prime } &=&\mu \alpha , \\
\beta ^{\prime } &=&\mu \left( b^{2}+c^{2}\right) \beta , \\
1 &=&\mu \varepsilon \left( b^{2}+c^{2}\right) \gamma .
\end{eqnarray*}

Now, by taking $b^{2}+c^{2}=1$ (for instance, $b=\cos \theta _{0}$ and $%
c=\sin \theta _{0}$ for some $\theta _{0}$), the third equation yields $\mu =%
\frac{\varepsilon }{\gamma }.$ Substituting the value of $\mu $ in the two
first equations, we deduce that $\alpha ^{\prime }=\varepsilon \frac{\alpha 
}{\gamma }$ and $\beta ^{\prime }=\varepsilon \frac{\beta }{\gamma }.$
Consequently, each $A_{4}\left( \alpha ,\beta ,\gamma \right) $ is
isomorphic to $A_{4}\left( \varepsilon \frac{\alpha }{\gamma },\varepsilon 
\frac{\beta }{\gamma },1\right) .$
\end{proof}

\bigskip

\textbf{Case 2.} $A_{3}=\left\langle e_{1},e_{2},e_{3}:e_{1}\cdot
e_{2}=e_{3},~e_{1}\cdot e_{3}=-e_{2},~e_{2}\cdot e_{2}=e_{3}\cdot
e_{3}=e_{1}\right\rangle .$

\smallskip

Similarly to the first case, we get 
\begin{equation*}
\delta _{1}h=\left( 
\begin{array}{ccc}
0 & h_{12} & h_{13} \\ 
0 & h_{22} & 0 \\ 
0 & 0 & h_{22}%
\end{array}%
\right) ,~\text{~}~\text{and ~}g=\left( 
\begin{array}{ccc}
0 & g_{12} & g_{13} \\ 
0 & g_{22} & g_{23} \\ 
0 & -g_{23} & g_{22}%
\end{array}%
\right) ,
\end{equation*}%
where $h_{12}=-h\left( e_{3}\right) ,~h_{13}=h\left( e_{2}\right)
,~h_{22}=-h\left( e_{1}\right) ,$ and $g_{ij}=g\left( e_{i},e_{j}\right) .$
It follows that, in this case, the class $\left[ g\right] \in H^{2}\left(
A_{3},\mathbb{R}\right) $ of a cocycle $g$ takes the reduced form

\begin{equation*}
g=\left( 
\begin{array}{ccc}
0 & 0 & 0 \\ 
0 & 0 & \gamma \\ 
0 & -\gamma & 0%
\end{array}%
\right) ,~~~\gamma \neq 0.
\end{equation*}

\bigskip By setting $\widetilde{e}_{i}=\left( e_{i},0\right) $, $1\leq i\leq
3,$ and $\widetilde{e}_{4}=\left( 0,1\right) ,$ and using formula (\ref%
{eq11bis}) we find that the nonzero relations are 
\begin{eqnarray}
\widetilde{e}_{1}\cdot \widetilde{e}_{2} &=&\widetilde{e}_{3},~\widetilde{e}%
_{1}\cdot \widetilde{e}_{3}=-\widetilde{e}_{2},~\widetilde{e}_{2}\cdot 
\widetilde{e}_{2}=\widetilde{e}_{3}\cdot \widetilde{e}_{3}=\widetilde{e}%
_{1}\   \label{eq16bis} \\
\widetilde{e}_{2}\cdot \widetilde{e}_{3} &=&\gamma \widetilde{e}_{4},\ \ 
\widetilde{e}_{3}\cdot \widetilde{e}_{2}=-\gamma \widetilde{e}_{4},  \notag
\end{eqnarray}%
with $\gamma \neq 0.$

\bigskip

We can now state the main result of this paper.

\begin{theorem}
\label{thm2}Let $A_{4}$ be a complete non-simple real left-symmetric algebra
whose associated Lie algebra is $\mathcal{O}\left( 4\right) .$ Then $A_{4}$
is isomorphic to one of the following left-symmetric algebras:

\begin{description}
\item[(i)] $A_{4}\left( s,t\right) :$ There exist real numbers $s,t$, and a
basis $\left\{ e_{1},e_{2},e_{3},e_{4}\right\} $ of $\mathcal{O}\left(
4\right) $ relative to which the nonzero left-symmetric relations are 
\begin{eqnarray*}
e_{1}\cdot e_{1} &=&se_{4},\ \ e_{2}\cdot e_{2}=e_{3}\cdot e_{3}=te_{4} \\
e_{1}\cdot e_{2} &=&e_{3},\ \ e_{1}\cdot e_{3}=-e_{2}, \\
e_{2}\cdot e_{3} &=&\frac{1}{2}e_{4},\ \ e_{3}\cdot e_{2}=-\frac{1}{2}e_{4}.
\end{eqnarray*}

\item The conjugacy class of $A_{4}\left( s,t\right) $ is given as follows: $%
A_{4}\left( s^{\prime },t^{\prime }\right) $ is isomorphic to $A_{4}\left(
s,t\right) $ if and only if $\left( s^{\prime },t^{\prime }\right) =\left(
\alpha s,\pm t\right) $ for some $\alpha \in \mathbb{R}^{\ast }.$

\item[(ii)] $B_{4}:$ There is a basis $\left\{
e_{1},e_{2},e_{3},e_{4}\right\} $ of $\mathcal{O}\left( 4\right) $ relative
to which the nonzero left-symmetric relations are%
\begin{eqnarray*}
e_{1}\cdot e_{2} &=&e_{3},\ \ e_{1}\cdot e_{3}=-e_{2},\ e_{2}\cdot
e_{2}=e_{3}\cdot e_{3}=e_{1} \\
e_{2}\cdot e_{3} &=&\frac{1}{2}e_{4},\ \ e_{3}\cdot e_{2}=-\frac{1}{2}e_{4}.
\end{eqnarray*}
\end{description}
\end{theorem}

\begin{proof}
According to the discussion above, there are two cases to be considered.

\textbf{Case 1.} $A_{3}=\left\langle e_{1},e_{2},e_{3}:e_{1}\cdot
e_{2}=e_{3},~e_{1}\cdot e_{3}=-e_{2}\right\rangle .$

In this case, Claim \ref{claim} asserts that $A_{4}$ is isomorphic to some $%
A_{4}\left( \alpha ,\beta ,1\right) $; and according to equations (\ref{eq16}%
), we know that there is a basis $\left\{ \widetilde{e}_{1},\widetilde{e}%
_{2},\widetilde{e}_{3},\widetilde{e}_{4}\right\} $ of $\mathcal{O}_{4}$
relative to which the nonzero relations for $A_{4}\left( \alpha ,\beta
,1\right) $ are: 
\begin{eqnarray*}
\widetilde{e}_{1}\cdot \widetilde{e}_{1} &=&\alpha \widetilde{e}_{4},\ 
\widetilde{e}_{2}\cdot \widetilde{e}_{2}=\widetilde{e}_{3}\cdot \widetilde{e}%
_{3}=\beta \widetilde{e}_{4}\  \\
\widetilde{e}_{1}\cdot \widetilde{e}_{2} &=&\widetilde{e}_{3},\ \ \widetilde{%
e}_{1}\cdot \widetilde{e}_{3}=-\widetilde{e}_{2}, \\
\widetilde{e}_{2}\cdot \widetilde{e}_{3} &=&\widetilde{e}_{4},\ \ \widetilde{%
e}_{3}\cdot \widetilde{e}_{2}=-\widetilde{e}_{4}.
\end{eqnarray*}%
Now, it is clear that by setting $s=\frac{\alpha }{2}$, $t=\frac{\beta }{2},$
$e_{i}=\widetilde{e}_{i}$ for $1\leq i\leq 3,$ and $e_{4}=2\widetilde{e}_{4},
$ we get the desired two-parameter family $A_{4}\left( s,t\right) .$

On the other hand, we see from equations (\ref{eq13}), (\ref{eq14}), and (%
\ref{eq15}) that $A_{4}\left( s^{\prime },t^{\prime }\right) $ is isomorphic
to $A_{4}\left( s,t\right) $ if and only if exists $\alpha \in \mathbb{R}%
^{\ast }$ and $b,c\in \mathbb{R}$, with $b^{2}+c^{2}\neq 0,$ such that 
\begin{eqnarray*}
s^{\prime } &=&\alpha s, \\
t^{\prime } &=&\alpha \left( b^{2}+c^{2}\right) t, \\
1 &=&\alpha \varepsilon \left( b^{2}+c^{2}\right) .
\end{eqnarray*}

From the third equation, we get $b^{2}+c^{2}=\frac{\varepsilon }{\alpha };$
and by substituting the value of $b^{2}+c^{2}$ in the second equation, we
get $t^{\prime }=\varepsilon t.$ In other words, we have shown that $%
A_{4}\left( s^{\prime },t^{\prime }\right) $ is isomorphic to $A_{4}\left(
s,t\right) $ if and only if exists $\alpha \in \mathbb{R}^{\ast }$ such that 
$s^{\prime }=\alpha s$ and $t^{\prime }=\pm t.\smallskip $

\textbf{Case 2.} $A_{3}=\left\langle e_{1},e_{2},e_{3}:e_{1}\cdot
e_{2}=e_{3},~e_{1}\cdot e_{3}=-e_{2},~e_{2}\cdot e_{2}=e_{3}\cdot
e_{3}=e_{1}\right\rangle .$

In this case, by (\ref{eq16bis}), there is a basis $\left\{ \widetilde{e}%
_{1},\widetilde{e}_{2},\widetilde{e}_{3},\widetilde{e}_{4}\right\} $ of $%
\mathcal{O}_{4}$ relative to which the nonzero relations in $A_{4}$ are:%
\begin{eqnarray*}
\widetilde{e}_{1}\cdot \widetilde{e}_{2} &=&\widetilde{e}_{3},~\widetilde{e}%
_{1}\cdot \widetilde{e}_{3}=-\widetilde{e}_{2},~\widetilde{e}_{2}\cdot 
\widetilde{e}_{2}=\widetilde{e}_{3}\cdot \widetilde{e}_{3}=\widetilde{e}%
_{1}\  \\
\widetilde{e}_{2}\cdot \widetilde{e}_{3} &=&\gamma \widetilde{e}_{4},\ \ 
\widetilde{e}_{3}\cdot \widetilde{e}_{2}=-\gamma \widetilde{e}_{4},
\end{eqnarray*}%
with $\gamma \neq 0.$

By setting $e_{i}=\widetilde{e}_{i}$ for $1\leq i\leq 3,$ and $e_{4}=2\gamma 
\widetilde{e}_{4},$ we see that $A_{4}$ is isomorphic to $B_{4}.$ This
finishes the proof of the main theorem. \bigskip
\end{proof}

\begin{remark}
Recall that a left-symmetric algebra $A$ is called \emph{Novikov} if it
satisfies the condition%
\begin{equation*}
\left( x\cdot y\right) \cdot z=\left( x\cdot z\right) \cdot y
\end{equation*}%
for all $x,y,z\in A.$ Novikov left-symmetric algebras were introduced in 
\cite{novikov} (see also \cite{zelmanov} for some important results
concerning this). We note here that $A_{4}\left( s,0\right) $ is Novikov and
that $B_{4}$ is not.
\end{remark}

\bigskip

We note that the mapping $X\mapsto \left( L_{X},X\right) $ is a Lie algebra
representation of $\mathcal{O}_{4}$ in $\mathfrak{aff}\left( \mathbb{R}%
^{4}\right) =End\left( \mathbb{R}^{4}\right) \oplus \mathbb{R}^{4}.$ By
using the (Lie group) exponential maps, Theorem \ref{thm2} can now be
stated, in terms of simply transitive actions of subgroups of the affine
group $Aff\left( \mathbb{R}^{4}\right) =GL\left( \mathbb{R}^{4}\right)
\ltimes \mathbb{R}^{4}$, as follows.

To state it, define the continuous functions $f,$ $g,$ $h,$ and $k$ by%
\begin{equation*}
f\left( x\right) =\left\{ 
\begin{array}{c}
\frac{\sin x}{x},\ \ \ x\neq 0 \\ 
1,\ \ \ \ x=0%
\end{array}%
\right. ,\ \ \ \ \ g\left( x\right) =\left\{ 
\begin{array}{c}
\frac{1-\cos x}{x},\ \ \ x\neq 0 \\ 
0,\ \ \ \ x=0%
\end{array}%
\right. ,
\end{equation*}%
and 
\begin{equation*}
h\left( x\right) =\left\{ 
\begin{array}{c}
\frac{x-\sin x}{x^{2}},\ \ \ x\neq 0 \\ 
0,\ \ \ \ x=0%
\end{array}%
\right. ,\ \ \ \ \ k\left( x\right) =\left\{ 
\begin{array}{c}
\frac{1-\cos x}{x^{2}},\ \ \ x\neq 0 \\ 
0,\ \ \ \ x=0%
\end{array}%
\right. ,
\end{equation*}%
and set 
\begin{eqnarray*}
\Phi _{t}\left( x\right) &=&\left( \frac{y}{2}+tz\right) g\left( x\right)
-\left( \frac{z}{2}-ty\right) f\left( x\right) , \\
\Psi _{t}\left( x\right) &=&\left( \frac{y}{2}+tz\right) f\left( x\right)
+\left( \frac{z}{2}-ty\right) g\left( x\right) .
\end{eqnarray*}

\begin{theorem}
\label{thm3}Suppose that the oscillator group $O_{4}$ acts simply
transitively by affine transformations on $\mathbb{R}^{4}.$ Then, as a
subgroup of $Aff\left( \mathbb{R}^{4}\right) =GL\left( \mathbb{R}^{4}\right)
\ltimes \mathbb{R}^{4},$ $O_{4}$ is conjugate to one of the following
subgroups:

\begin{description}
\item[(i)] 
\begin{equation*}
G_{4}=\left\{ 
\begin{array}{c}
\left[ 
\begin{array}{cccc}
1 & yf\left( x\right) +zg\left( x\right) & zf\left( x\right) -yg\left(
x\right) & 0 \\ 
0 & \cos x & -\sin x & 0 \\ 
0 & \sin x & \cos x & 0 \\ 
0 & \Phi _{0}\left( x\right) & \Psi _{0}\left( x\right) & 1%
\end{array}%
\right] \\ 
~\times \left[ 
\begin{array}{c}
x+\left( y^{2}+z^{2}\right) k\left( x\right) \\ 
yf\left( x\right) -zg\left( x\right) \\ 
zf\left( x\right) +yg\left( x\right) \\ 
w+\frac{\left( y^{2}+z^{2}\right) }{2}h\left( x\right)%
\end{array}%
\right] :\ x,y,z,w\in \mathbb{R}%
\end{array}%
\right\} ,
\end{equation*}

\item[(ii)] 
\begin{equation*}
G_{4}\left( s,t\right) =\left\{ 
\begin{array}{c}
\left[ 
\begin{array}{cccc}
1 & 0 & 0 & 0 \\ 
0 & \cos x & -\sin x & 0 \\ 
0 & \sin x & \cos x & 0 \\ 
sx & \Phi _{t}\left( x\right) & \Psi _{t}\left( x\right) & 1%
\end{array}%
\right] \times \\ 
\left[ 
\begin{array}{c}
x \\ 
yf\left( x\right) -zg\left( x\right) \\ 
zf\left( x\right) +yg\left( x\right) \\ 
w+\frac{s}{2}x^{2}+\left( y^{2}+z^{2}\right) \left( \frac{h\left( x\right) }{%
2}+tk\left( x\right) \right)%
\end{array}%
\right] :x,y,z,w\in \mathbb{R}%
\end{array}%
\right\} ,
\end{equation*}%
where $s,t\in \mathbb{R}.$ The only pairs of conjugate subgroups in $%
Aff\left( \mathbb{R}^{4}\right) $ are $G_{4}\left( s,t\right) $ and $%
G_{4}\left( \alpha s,\pm t\right) $ where $\alpha \in \mathbb{R}^{\ast }.$
\end{description}
\end{theorem}

\bigskip

Department of Mathematics,

College of Science,

King Saud University,

P.O.Box 2455, Riyadh 11451

Saudi Arabia

\smallskip

E-mail: mguediri@ksu.edu.sa

\end{document}